\documentclass{article}
\usepackage{amsmath,amssymb,amsfonts,amsthm}
\usepackage{bm,graphicx,comment,color,url}
\usepackage{algorithmic,algorithm}
\usepackage{mathscinet}

\usepackage{cancel}
\usepackage[normalem]{ulem}

\newtheorem{theorem}{Theorem}[section]

\newtheorem{lemma}[theorem]{Lemma}
\newtheorem{proposition}[theorem]{Proposition}

\newcommand{\cu}{\mathfrak i}

\usepackage{graphicx}
\usepackage{epstopdf}
\usepackage[misc]{ifsym}
\usepackage{amssymb}
\usepackage{amsmath}
\usepackage{epsfig}
\usepackage{a4wide} 
\usepackage{color}
\usepackage{booktabs}
\usepackage[all]{xy}
\usepackage{latexsym}
\usepackage{graphics}
\usepackage{multirow}
\usepackage{tabularx}
\usepackage{longtable}
\usepackage{multirow}
\usepackage{makecell}
\usepackage{mathrsfs}
\usepackage{delarray}
\usepackage{subfig}

\begin{document}
\title{
Computing eigenvalues of semi-infinite quasi-Toeplitz matrices\thanks{This work has been partially supported by University of Pisa’s project PRA\_2020\_61, and by GNCS of INdAM}
}
\author{D.A. Bini\thanks{University of Pisa, Italy}, B. Iannazzo\thanks{University of Perugia, Italy}, B. Meini\thanks{University of Pisa, Italy}, 
J. Meng\thanks{Ocean University of China, Qingdao, Shandong, China}, L. Robol\thanks{University of Pisa, Italy}
}
\maketitle

\begin{abstract}
 A quasi-Toeplitz (QT) matrix is a semi-infinite matrix of the form
 $A=T(a)+E$ where $T(a)$ is the Toeplitz matrix with entries
 $(T(a))_{i,j}=a_{j-i}$, for $a_{j-i}\in\mathbb C$, $i,j\ge 1$, while
 $E$ is a matrix representing a compact operator in $\ell^2$.  The
 matrix $A$ is finitely representable if $a_k=0$ for $k<-m$ and for
 $k>n$, given $m,n>0$, and if $E$ has a finite number of nonzero
 entries.  The problem of numerically computing eigenpairs
 of a finitely representable QT matrix is investigated, i.e., pairs
  $(\lambda,\bm v)$ such that $A\bm v=\lambda \bm v$, with $\lambda\in\mathbb C$, $\bm
 v=(v_j)_{j\in\mathbb Z^+}$, $\bm v\ne 0$, and $\sum_{j=1}^\infty |v_j|^2<\infty$. It is shown
 that the problem is reduced to a finite nonlinear eigenvalue problem
 of the kind $ WU(\lambda)\bm \beta=0$, where $W$ is a constant matrix and
 $U$ depends on $\lambda$ and can be given in terms of either a
 Vandermonde matrix or a companion matrix.  Algorithms relying on
 Newton's method applied to the equation $\det WU(\lambda)=0$ are
 analyzed. Numerical experiments show the effectiveness of this
 approach.  The algorithms have been included in the CQT-Toolbox
 [Numer. Algorithms 81 (2019), no. 2, 741--769].
\end{abstract}

\section{Introduction}
A quasi-Toeplitz (QT) matrix $A$ is a
semi-infinite matrix that can be written as $A=T(a)+E$ where
$T(a)=(t_{i,j})_{i,j\in\mathbb Z^+}$ is Toeplitz, i.e., $t_{i,j}=a_{j-i}$ for a given
sequence $\{a_k\}_{k\in\mathbb Z}$, and $E$ is compact, that is, $E$
is the limit of a sequence of semi-infinite matrices $E_i$ of finite
rank. Here, convergence means that $\lim_{i\to\infty}\|E-E_i\|_s=0$, 
where $\|\cdot\|_s$ is the operator norm induced by the vector norm
$\|\bm v\|_s=(\sum_{i=1}^\infty |v_i|^s)^\frac1s$, for $\bm
v=(v_i)_{i\in\mathbb Z^+}$, and the value of $s\ge 1$ depends on the
specific context where the mathematical model originates.

Matrices of this kind are encountered in diverse applications related
to semi-infinite domains. For instance, the analysis of queuing
models, where buffers have infinite capacity, leads to QT matrices
where the compact correction reproduces the boundary conditions of the
model while the Toeplitz part describes the inner action of the
stochastic process. A typical paradigm in this framework is given by
random walks in the quarter plane. Some references in this regard can
be found in the books \cite{blm:book}, \cite{lr:book}, \cite{neuts},
and in the more recent papers \cite{jackson}, \cite{ozawa19},
\cite{ozawa}. Another classical and meaningful example concerns the
class of matrices that discretize boundary value problems by means of
finite differences. In this case, the Toeplitz part of the QT matrix
describes the inner action of the differential operator, while the
compact correction expresses the boundary conditions imposed on the
differential system. In this framework, it is worth citing the two
books \cite{gs2:book}, \cite{gs1:book}, that are a relevant reference on a
very close subject concerning  generalized locally Toeplitz
matrices (GLT) and their applications, where a rich literature is cited.

Computational aspects in the solution of matrix equations with QT
matrices in bidimensional random walk have been recently investigated
in \cite{bmm}, \cite{bmmr}, \cite{bmm20}, while generalizations
including probabilistic models with restarts are analyzed in
\cite{bmmr:sicomp}. Other applications of QT matrices have been
considered in \cite{mean1}, \cite{mean2}, \cite{bm:numermath},
concerning matrix functions and means, and in
\cite{jie}, \cite{leonardo} concerning Sylvester equations.  Important
sources of theoretical properties of QT matrices are given in the
books \cite{BG:book2000}, \cite{BG:book2005}, and \cite{BS:book}.  In
\cite{bmr} a suitable Matlab toolbox, the CQT-toolbox, has been
introduced for performing arithmetic operations with QT matrices
including the four arithmetic operations and the more relevant matrix
factorizations.

\subsection{Main results}
In this paper, we deal with the computation of the eigenvalues of  QT
matrices, a topic that was not covered in the CQT-Toolbox of \cite{bmr}. Namely,
we are interested in the design and analysis of algorithms for
computing the eigenvalues $\lambda$ and the corresponding eigenvectors
$\bm v$ of a given QT matrix $A$, that is, $\bm v$ is such that $A\bm v=\lambda
\bm v$ and $\bm v\in\ell^s$ where $1\le s<\infty$. Here $\ell^s$ is
the set of vectors $\bm x=(x_i)_{i\ge 1}$ such that $\|\bm
x\|_s<\infty$.  For the sake of simplicity, in the following we will
set $s=2$ and use $\|\cdot\|$ to denote the $2$-norm.  The attention
is restricted to the case where $A$ is finitely representable, i.e.,
$A=T(a)+E$, where $T(a)$ is a band Toeplitz matrix determined by a finite
number of parameters $a_{-m},\ldots,a_{n}$ for $m,n>0$, $E$ is a
matrix having infinitely many rows and columns but a finite number of
nonzero entries.  A matrix of this kind represents a bounded linear
operator from $\ell^2$ in $\ell^2$.  We associate with the matrix $A$
the Laurent polynomial $a(z)=\sum_{i=-m}^na_iz^i$.

 Recall that the spectrum of a bounded operator $A$ is the set of
 $\lambda\in\mathbb C$ such that $A-\lambda I$ is not invertible, and
 the essential spectrum is the set of $\lambda\in\mathbb C$ such that
 $A-\lambda I$ is not Fredholm.  We wish to point out that,
 not all the points of the spectrum or of the essential spectrum are
 necessarily eigenvalues of $A$. Moreover, while for a Toeplitz matrix
 $A$ the set of eigenvalues does not contain isolated points and can
 be explicitly determined by the image $a(\mathbb T)$ of the unit
 circle $\mathbb T$ through the Laurent polynomial $a(z)$ and by the
 winding number of $a(z)-\lambda$ (see \cite{BG:book2000}), for a
 general QT matrix having a nontrivial compact correction the set of
 eigenvalues may contain a continuous part and a discrete part, the
 latter is formed by a set of isolated eigenvalues. As an example, see
 Figure \ref{fig:1}.

\begin{figure}
\centering
\includegraphics[scale=0.4]{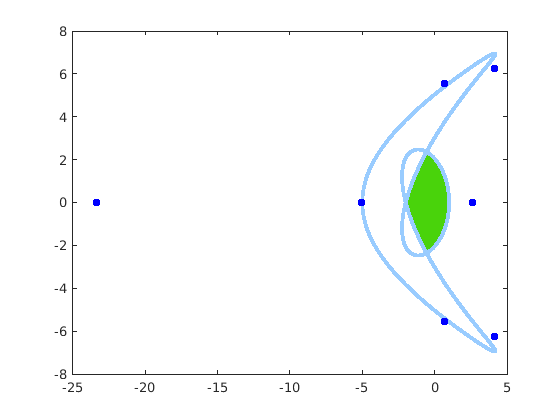}
\caption{\footnotesize Isolated eigenvalues (blue dots) and dense set
  of eigenvalues (green area) of the QT matrix associated with
  $a(z)=z^{-3}-3z^{-2}+2z^{-1}-2z+z^2+2z^3$ and with the
  correction $E=(e_{i,j})$, $e_{1,j}=-2j$, $e_{2,j}=-2(10-j)$,
  $e_{3,j}=-2$, for $j=1,\ldots,9$, $e_{i,j}=0$
  elsewhere.}\label{fig:1}
\end{figure}

We prove that any isolated eigenvalue $\lambda$ of a QT matrix $A$ is
the solution of a finite nonlinear eigenvalue problem of the form
\[
W U(\lambda)\bm \gamma=0,\quad \bm \gamma\in\mathbb C^p\setminus \{0\} ,
\]
where $W$ is a $q\times k$ constant matrix and $U(\lambda)$ is a
$k\times p$ matrix-valued function whose size $p$ and entries depend
on $\lambda$ in an implicit way. Here $k,q>0$ are integers depending
on the given matrix $A$, while $p$ is the number of zeros $\xi_j$,
$j=1,\ldots,p$ of modulus less than 1 of the Laurent polynomial
$a(z)-\lambda$. It is well-known that the value of $p$ is given by
$p=m+w$ where $w$ is the winding number of $a(z)-\lambda$. Thus, it
takes constant values on each connected component $\Omega$ of the set
$\mathbb C\setminus a(\mathbb T)$ (see Figure \ref{fig:conncomp} for
an example). Note that while $p$ depends on $\lambda$, it is locally
constant on $\mathbb C\setminus a(\mathbb T)$, and thus we will not
write explicitly the dependence on $\lambda$.

We consider two different forms of $U=(u_{i,j})$: the {\em Vandermonde
  version} and the {\em Frobenius version}. In the former version, $U$
can be chosen as the Vandermonde matrix with entries
$u_{i,j}=\xi_j^{i-1}$, $i=1,\ldots,k$, $j=1,\ldots,p$, provided that
$\xi_i\ne\xi_j$ for $i\ne j$. In the latter, $U$ is the truncation to
size $k\times p$ of the matrix $[I;G;G^2;\ldots]$ (we adopted the
Matlab notation where ``;'' separates block rows of the matrix), where
$G=F^p$ is the $p$-th power of the $p\times p$ companion (Frobenius)
matrix $F$ associated with the monic polynomial
$s(z)=\prod_{j=1}^p(z-\xi_j)=z^p+\sum_{j=0}^{p-1} s_jz^j$.

This formulation of the problem allows us to detect those components
$\Omega$ that constitute continuous sets of eigenvalues (for $q<p$),
and to design numerical algorithms for computing the isolated
eigenvalues of $A$ (for $q\ge p$) by solving the corresponding
nonlinear eigenvalue problem.  Nonlinear eigenvalue problems have
recently received much attention in the literature. Here we refer to
the survey paper \cite{gt-2017}, to the subsequent paper
\cite{feast-2018}, to the more recent works \cite{gander-2021} and
\cite{hp-2020}, and to the references there in.

Our algorithms follow the classical approach of applying Newton's
iteration, as done in \cite{gander-2021} and
\cite{gt-2017}, to the scalar equation $f(\lambda)\!=\!0$, where
$f(\lambda)\!=\!\det (WU(\lambda))$ by relying on the Jacobi identity
$f(\lambda)/f'(\lambda)\!=\!1/\hbox{trace}((WU(\lambda))^{-1}WU'(\lambda))$.  Here, the main problem is to exploit the specific
features of the function $f(\lambda)$ through the design of efficient
algorithms to compute $U(\lambda)$ and $U'(\lambda)$ in both the
Vandermonde and in the Frobenius formulation. This analysis leads to
the algorithmic study of some interesting computational problems such
as computing the winding number of $a(z)-\lambda$, or computing the
coefficients of the polynomial factor $s(z)$ having zeros of modulus
less than 1 together with their derivatives with respect to $\lambda$,
or computing $G=F^p$ and the derivative of $G^j$ for $j=1,2,\ldots$,
with respect to $\lambda$.  We will accomplish the above tasks by
relying on the combination of different computational tools such as
Graeffe's iteration \cite{graeffe}, the Wiener-Hopf factorization of
$a(z)-\lambda$ computed by means of the cyclic reduction algorithm
\cite{blm:book}, and the Barnett factorization of $F^p$
\cite{barnett}.

The algorithms based on the Vandermonde and on the Frobenius versions
require either the computation of the zeros of the Laurent polynomial
$a(z)-\lambda$ and the selection of those zeros $\xi_1,\ldots,\xi_p$
having modulus less than 1, or the computation of the coefficients of
the factor $\prod_{j=1}^p(z-\xi_j)$. In principle, the latter approach
is less prone to numerical instabilities and avoids the theoretical
difficulties encountered when there are multiple or clustered zeros.
This fact is confirmed by numerical tests and analysis.

Our procedure uses Newton's iteration as an effective tool for
refining a given approximation to an eigenvalue. In order to
numerically compute all the eigenvalues we have combined Newton's
iteration with a heuristic strategy based on choosing as starting
approximations the eigenvalues of the $N\times N$ matrix $A_N$ given
by the leading principal submatrix of $A$ of sufficiently large
size. In fact, we may show that for any $\epsilon>0$, the
$\epsilon$-pseudospectrum of $A_N$ gets closer to any isolated
eigenvalue of $A$ as $N$ gets large.

One could argue that a large value of $N$ would provide an
approximation of the isolated eigenvalues of $A$,
directly. Nevertheless, our approach requires only a rough
approximation of the isolated eigenvalues and thus a smaller value of
$N$, followed by Newton's iteration, to compute the eigenvalues with
the same accuracy.  Numerical experiments show the effectiveness of
this approach: examples are shown where in order to obtain
full-precision approximations of the eigenvalues of $A$ from the
eigenvalues of $A_N$ would require large values of $N$ (of the order
of millions), while starting Newton's iteration with the eigenvalues
of $A_N$ for moderate values of $N$ (of the order of few hundreds)
provides very accurate approximations in few steps.

\subsection{Paper organization}
The paper is organized as follows. In Section 2 we recall some
preliminary properties that are useful in the subsequent analysis.  In
particular, Section 2.1 deals with the eigenvalues of $T(a)$ while
Section 2.2 deals with the eigenvalues of $T(a)+E$. Section 3 concerns
the reduction of the original eigenvalue problem for QT operators to
the form of a nonlinear eigenvalue problem for finite matrices in the
Frobenius and in the Vandermonde versions. Section 4 concerns further
algorithmic issues. In particular, an efficient method for computing
the winding number of a Laurent polynomial is designed based on the
Graeffe iteration; the problem of computing a factor of the polynomial
$a(z)-\lambda$ together with its derivative with respect to $\lambda$
is analyzed relying on the Barnett factorization and on the solution
of a linear system associated with a resultant matrix; morever, in the
same section we prove the regularity of the function
$\det(WU(\lambda))$ to which Newton's iteration is applied. In Section
5 we investigate on the relationships between the isolated eigenvalues
of $A$ and the eigenvalues of $A_N$ when $N$ gets large.  The results
of some numerical experiments are reported in Section 6. Finally,
Section 7 draws the conclusions and describes some open problems.

The algorithms, implemented in Matlab, have been added to the
CQT-Toolbox of \cite{bmr}. The main functions are \verb+eig_single+
and \verb+eig_all+. The former computes a single eigenvalue of a QT
matrix starting from a given approximation, and, optionally, an
arbitrary number of components of the corresponding eigenvector, the
latter provides the computation of all the eigenvalues. Other related
functions integrate the package.  More information, together with the
description of other auxiliary functions and optional parameters can
be found at \url{https://numpi.github.io/cqt-toolbox} while the
software can be downloaded at
\url{https://github.com/numpi/cqt-toolbox}.

\section{Preliminaries}
Let $a(z)=\sum_{i=-m}^n a_iz^i$ be a Laurent polynomial where $a_i\in\mathbb C$ for $i=-m,\ldots,n$, and
$a_{-m},a_n\ne 0$. Define $T(a)=(t_{i,j})_{i,j=1,2,\ldots}$,
$t_{i,j}=a_{j-i}$, the Toeplitz matrix associated with
$a(z)$. Given a semi-infinite matrix $E=(e_{i,j})_{i,j=1,2,\ldots}$,
such that $e_{i,j}=0$ for $i>k_1$, or $j>k_2$, the matrix $A = T(a)+E$
represents a bounded linear operator from the set
$\ell^2=\{(v_i)_{i\in\mathbb Z^+}\,:\, v_i\in\mathbb C,
~\sum_{i=1}^\infty |v_i|^2<\infty\}$ to itself. Denote by $\mathcal
B(\ell^2)$ the set of bounded linear operators from $\ell^2$ to itself
and by $\mathbb T$ the unit circle in the complex plane.

Recall that $A$ is invertible if there exists $B\in\mathcal B(\ell^2)$
such that $AB=BA=I$, where $I$ is the identity on $\mathcal
B(\ell^2)$. Moreover, $A$ is Fredholm if there exists $B\in\mathcal
B(\ell^2)$ such that $AB-I$ and $BA-I$ are compact, i.e., $A$ is
invertible modulo compact operators.  Recall also that for
$A\in\mathcal B(\ell^2)$ the spectrum of $A$ is defined as
\[
\hbox{sp}(A)=\{\lambda\in\mathbb C\,:\, A-\lambda I\hbox{ is not invertible}\}
\]
while the essential spectrum is defined as
\[
\hbox{sp}_{\rm ess}(A)=\{\lambda\in\mathbb C\,:\, A-\lambda I\hbox{ is not Fredholm}\},
\]
so that $\hbox{sp}_{\rm ess}(A)\subset \hbox{sp}(A)$.

It is well known that for a Laurent polynomial $a(z)$, $T(a)$ is
invertible in $\ell^s$ if and only if $a(z)\ne 0$ for $z\in\mathbb T$
and $\hbox{wind}(a)=0$ (see \cite[Corollary 1.11]{BG:book2005}), where
$\hbox{wind}(a)$ is the winding number of the curve $a(\mathbb T)$.

In the case where $a(z)$ is a Laurent polynomial, we may write
\begin{equation}\label{eq:wind}
\hbox{wind}(a)=\frac1{2\pi}\int_0^{2\pi}e^{\cu t}\frac{a'(e^{\cu  t})}{a(e^{\cu t})}dt,
\end{equation}
where $a'(z)=\sum_{j=-m}^n ja_{j}z^{j-1}$ is the first derivative of
$a(z)$. Notice that $\hbox{wind}(a-\lambda)$ is constant for $\lambda$
in each connected component $\Omega$ of the set $\mathbb C\setminus
a(\mathbb T)$.  Consequently, we have (see \cite[Corollary 1.12]{BG:book2005})
\begin{equation}\label{eq:spect}
\hbox{sp}(T(a))=a(\mathbb T)\cup\{\lambda\in\mathbb C\setminus a(\mathbb T)\,:\,\hbox{wind}(a-\lambda)\ne 0\},
\end{equation}
  moreover, 
\begin{equation}\label{eq:spess}
\hbox{sp}_{\rm ess}(T(a))=a(\mathbb T).
\end{equation}
We say that $(\lambda,\bm v)$ is an eigenpair (eigenvalue, eigenvector) if $A\bm v=\lambda \bm v$ and $\bm v\in\ell^2$.

\subsection{Eigenvalues of $T(a)$}
The following results from \cite{BG:book2005} characterize the
eigenpairs of the Toeplitz operator $T(a)$. In this statement and
throughout the paper, we used a slightly different notation with
respect to \cite{BG:book2005}. Namely, we denote the entries of $T(a)$
as $(T(a))_{i,j}=a_{j-i}$, while in the classical literature they are
denoted as $(T(a))_{i,j}=a_{i-j}$. The reason is that this notation is
more suited to fit the context of Markov chains and queueing models
where these matrices play an important role.

\begin{lemma}\label{eig:onT}\cite[Proposition 1.20]{BG:book2005}
Let $1\leq s\leq \infty$. For a Laurent polynomial $a(z)$, a point
$\lambda\notin a(\mathbb{T})$ is an eigenvalue of $T(a)$ as an
operator on $\ell^s$ if and only if $r:={\rm wind}(a-\lambda)>0$.
Moreover, the kernel of $T(a)-\lambda I$ has dimension $r$ and if $\bm
v\in\hbox{\rm ker}(T(a)-\lambda I)$ then $\bm v$ is exponentially
decaying.
\end{lemma}

If $\lambda\in a(\mathbb T)$ a similar result can be given. Let
$\tau_1,\ldots,\tau_q$ be the distinct zeros of $a(z)-\lambda$ of
modulus 1 and multiplicity $\alpha_1,\ldots,\alpha_q$,
respectively. Define
\begin{equation}\label{eq:c}
    c(z)= (a(z)-\lambda)/\prod_{j=1}^q \Bigl(1-\frac z{\tau_j}\Bigr)^{\alpha_j},
\end{equation}
 so that $c(z)$ is a Laurent polynomial having no zero on $\mathbb T$. Then we have the following.

\begin{lemma}\label{eig:onT1}\cite[Proposition 1.22]{BG:book2005}
Let $1\leq s\leq \infty$. For a Laurent polynomial $a(z)$, a point
$\lambda\in a(\mathbb{T})$ is an eigenvalue of $T(a)$ as an operator
on $\ell^s$ if and only if $r:={\rm wind}(c)>0$.  Moreover, the kernel
of $T(a)-\lambda I$ has dimension $r$ and if $\bm v\in\hbox{\rm
  ker}(T(a)-\lambda I)$ then $\bm v$ is exponentially decaying.
\end{lemma}

Observe that according to the above lemmas, the eigenvalues of $T(a)$ belong to the set $\hbox{sp}(T(a))$, that, in turn, can be explicitly described by means of \eqref{eq:spect}.

Let $\Omega$ be a connected component of the set $\mathbb C\setminus
a(\mathbb T)$. The function $\hbox{wind}(a-\lambda)$ is constant on
$\Omega$, and this means that if the winding number is $r>0$ then all
the values $\lambda\in\Omega$ are eigenvalues of $T(a)$ of (geometric)
multiplicity $r$, while if $r\le 0$ then no $\lambda\in\Omega$ is
eigenvalue of $T(a)$. We recall Proposition 1.25 from
\cite{BG:book2005}.

\begin{figure}
\begin{center}
\includegraphics[width=6cm]{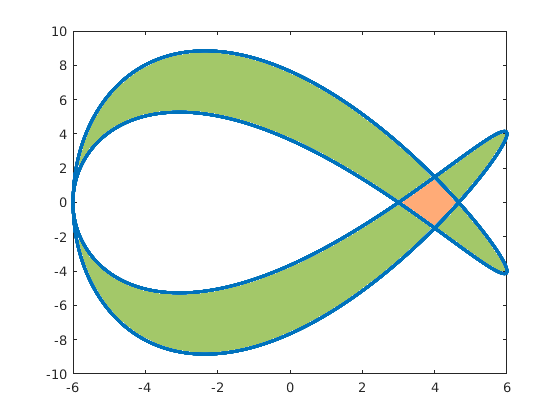}
\end{center}\caption{\footnotesize Connected components of
  $\mathbb C\setminus a(\mathbb T)$ for the Laurent polynomial
  $a(z)=3z^{-3}-2z^{-2}+z^{-1}-z-4z^2-3z^3$. In green and orange the
  components with winding numbers 1 and 2, respectively. In white the
  components with winding number 0.}\label{fig:conncomp}
\end{figure}

\begin{lemma}
If $\lambda\in a(\mathbb T)$ is in the boundary of $\Omega$, and $\xi
= \hbox{\rm wind}(a-\mu)$ for $\mu\in\Omega$, then $\xi \ge \hbox{\rm
  wind}(c)$, where $c(z)$ is defined in \eqref{eq:c}.
\end{lemma}

From the above results it follows that (compare with Corollary 1.26 in
\cite{BG:book2005}) if $\lambda$ lies on the boundary of $\Omega$ such
that $\hbox{wind}(a-\mu)\le 0$ for $\mu\in\Omega$ then $\lambda$
cannot be an eigenvalue of $T(a)$. That is, the eigenvalues of $T(a)$
belong necessarily to those components $\Omega$ for which
$\hbox{wind}(a-\lambda)>0$ and to their boundaries. Therefore $T(a)$
cannot have isolated eigenvalues.

\subsection{Eigenvalues of $T(a)+E$}
From the definition of spectrum and of essential spectrum it follows that
\[
\hbox{sp}_{\rm ess}(T(a))=\hbox{sp}_{\rm ess}(T(a)+E)\subset \hbox{sp}(T(a)+E)
\]
for any compact operator $E$. In fact, to prove the equality, if
$A=T(a)+E$ is a QT matrix, then $A-\lambda I$ is not Fredholm iff
$B(A-\lambda I)-I$ and $(A-\lambda I)B-I$ are not compact for any
bounded operator $B$. That is, iff $B(T(a)-\lambda I)-I+BE$ and
$(T(a)-\lambda I)B-I+EB$ are not compact. This is equivalent to say
that $B(T(a)-\lambda I)-I$ and $(T(a)-\lambda I)B-I$ are not compact,
i.e., $T(a)-\lambda I$ is not Fredholm.

Another interesting property is given by the following.

\begin{proposition}
If $A=T(a)+E$ is a QT matrix, then ${\rm sp}(T(a))\subset {\rm sp}(A)$. 	
\end{proposition}

The above result is an immediate consequence of the following

\begin{lemma}\label{lem:invertibility}
  If the QT matrix $A=T(a)+E$ is invertible on
  $\ell^s \ (1\leq s\leq \infty)$ , then $T(a)$ is also invertible on $\ell^s$.
\end{lemma}
\begin{proof}
  Since $A=T(a)+E$ is invertible, then $0\notin {\rm sp}(A)$. This
  implies $0\notin {\rm sp_{ess}}(A)={\rm
    sp_{ess}}(T(a))=a(\mathbb{T})$ so that $a(z)\ne 0$ for all $z\in
  \mathbb{T}$. To show $T(a)$ is invertible, it is sufficient to show
  that the winding number of $a(z)$ is 0, that is,
  $\hbox{wind}(a)=0$. To this end, suppose wind$(a)=m$ and $m\ne 0$,
  then $A$ is a Fredholm operator and it follows from \cite[Theorem
    2.8]{Schechter} and \cite[Theorem 1.9]{BG:book2005} that the index
  of $A$ is ${\rm Ind}\ A={\rm Ind}\ T(a)=m\ne 0$.  On the other hand,
  since $A$ is invertible, it follows that $ {\rm dim \ Ker}\ A ={\rm
    dim\ Coker} \ A=0, $ where ${\rm dim\ Ker} \ A$ is the dimension
  of the kernel of $A$ and ${\rm dim\ Coker}\ A$ is the dimension of
  the cokernel of $A$. It follows from \cite[page 9]{BG:book2005} that
  the index of A is ${\rm Ind}\ A={\rm dim\ Ker}\ A- {\rm
    dim\ Coker}\ A=0$, from which we get a contradiction. Hence, we
  must have wind$(a)=0$.
\end{proof}

Observe that, in general, $\lambda\in\hbox{sp}(T(a)+E)$ does not imply
$\lambda\in\hbox{sp}(T(a))$, as the following example shows.  Denote
by $\hbox{trid}(\alpha,\beta,\gamma)$ the tridiagonal Toeplitz matrix
associated with the Laurent polynomial $\alpha z^{-1}+\beta+\gamma z$.
Let $T(a)=\hbox{trid}(-2,5,-2)$, so that $T(a)=UU^T$,
$U=\hbox{trid}(0,2,-1)$. Set $A=T(a)-4\bm e_1 \bm e_1^T$, where
$\bm e_1=[1,0,\ldots]^T$, so that
$A=U\hbox{diag}(0,1,1,\ldots)U^T$. Then $0\in\hbox{sp}(A)$
since $A$ is not invertible (but it is Fredholm), while $0\not
\in \hbox{sp}(T(a))$ since $T(a)$ is invertible being $U$ and $U^T$
invertible operators.  That is, adding a compact correction $E$ to
$T(a)$ there may be eigenvalues of $A=T(a)+E$ not belonging to
$\hbox{sp}(T(a))$.

The following two results are useful for our analysis.

\begin{proposition}\label{prop:1}
   Let $\lambda\notin a(\mathbb T)$ and $w={\rm
     wind}(a-\lambda)$. Then the Laurent polynomial $a(z)-\lambda$ has
   $p=m+w$ zeros of modulus less than 1.
\end{proposition}

\begin{proof}
  Since $a(e^{\cu t})' = \cu e^{\cu t}a'(e^{\cu t})$, from
  \eqref{eq:wind} we get $\frac{1}{2\pi \mathfrak
    i}\int_0^{2\pi}\frac{a(e^{\mathfrak i t})'}{a(e^{\mathfrak i
      t})-\lambda}dt=w$, which implies that the number $p$ of zeros
  and the number $\hat m$ of poles of $a(z)-\lambda$ in the open unit
  disk are such that $p-\hat m=w$. Since $\hat m=m$, it follows
  $p=m+w$.
\end{proof}

A similar result holds for $\lambda \in a(\mathbb T)$.

\begin{proposition}\label{prop:2}
	Let $\lambda\in a(\mathbb{T})$ and suppose that $a(z)-\lambda$
        has $q$ zeros $\tau_1,\ldots, \tau_q$ of modulus 1 with
        multiplicities $\alpha_1,\ldots, \alpha_q$, let $c(z)$ be the
        Laurent polynomial in \eqref{eq:c}.  Then, $a(z)-\lambda$ has
        $p=m+w-(\alpha_1+\ldots+\alpha_q)$ zeros of modulus less than
        1, where $w={\rm wind}(c)$.
\end{proposition}

\section{Computational analysis}\label{sec:aa}
In this section, we aim at the design and analysis of numerical
algorithms for computing the eigenvalues of thefinitely representable  QT matrix $A=T(a)+E$
belonging to a given connected component $\Omega$ of $\mathbb C
\setminus a(\mathbb T)$, together with the corresponding eigenvectors.
For the sake of simplicity, the case $\lambda\in a(\mathbb T)$ is not
treated in this paper.

If $E=0$ then the spectrum and the essential spectrum of $T(a)$ are
explicitly known (see \eqref{eq:spect}, and
\eqref{eq:spess}). Moreover, an eigenvalue $\lambda$, together with
its multiplicity, can be explicitly characterized in terms of the
winding number $\hbox{wind}(a-\lambda)$, if $\lambda\notin a(\mathbb
T)$ (see Lemma \ref{eig:onT}).
Therefore the case of interest is $E\ne
0$.

Recall the following notations: $a(z)=\sum_{i=-m}^na_iz^i$, while
$k_1$ is the row size of the non zero part of the correction $E$. We
set $q=\max(m,k_1)$, and denote $p$ the number of zeros of modulus
less than 1 of the Laurent polynomial $a(z)-\lambda$. In view of
Proposition \ref{prop:1} we have $p=m+\hbox{wind}(a-\lambda)$,
moreover $p$ is constant for $\lambda\in\Omega$. Finally, for a given
matrix $A$, we denote by $A_{r\times s}$ the leading principal
submatrix of $A$ of size $r\times s$, i.e., the submatrix formed by
the entries in the first $r$ rows and in the first $s$ columns. If
$r=s$ we write $A_r$ in place of $A_{r\times s}$.

\subsection{Reduction to a nonlinear eigenvalue problem}

Consider an eigenpair $(\lambda,\bm v)$ of $A=T(a)+E$ so that $\bm u:=(A-\lambda I)\bm v=0$. Observe that the condition $u_k=0$ for $k\ge q+1$ can be written as the linear difference equation
 \begin{equation}\label{eq:diffeq}
 \sum_{j=-m}^n a_j v_{k+j}-\lambda  v_k=0,\quad k\ge q+1, 
 \end{equation}
whose characteristic polynomial is $b(z)=(a(z)-\lambda)z^m$. The
dimension of the space of solutions of \eqref{eq:diffeq} that belong
to $\ell^2$ depends on $\lambda$ and coincides with the number $p$ of
roots of $a(z)-\lambda$ with modulus less than $1$. Our two approaches
differ in the way the basis of the latter space is chosen.

If $\bm v^{(j)}$, $j=1,\ldots,p$ is a basis of the space of solutions,
then we may write the eigenvector $\bm v$ as a linear combination of
$\bm v^{(j)}$, i.e., $\bm v=\sum_{j=1}^p \alpha_j \bm v^{(j)}$.
Therefore, we may say that $(\lambda,\bm v)$ is an eigenpair for $A$
if and only if $\bm v=\sum_{j=1}^p\alpha_j \bm v^{(j)}$ and the
conditions $u_1=\ldots = u_q=0$ are satisfied.
 
The latter conditions form a nonlinear system in $q$ equations and $p$
unknowns which can be written as
\begin{equation}\label{eq:sys}
\begin{aligned}
&HV(\lambda){\bm \alpha}=\lambda V_{q\times p}(\lambda)\bm\alpha,\quad H=A_{q\times \infty}\\
& V(\lambda)=[\bm v^{(1)},\bm v^{(2)},\ldots,\bm v^{(p)}],\quad \bm\alpha\in\mathbb C^p,\quad \hbox{wind}(a-\lambda)=p-m.
\end{aligned}
\end{equation}
In fact, $\lambda$ and the $p$ components of $\bm \alpha$, normalized
such that $\|\bm\alpha\|=1$, form a set of $p$ unknowns.  It is clear
that the system \eqref{eq:sys} is in the form of a nonlinear
eigenvalue problem (NEP).

This system has a non-trivial solution $\bm\alpha$ for a given
$\lambda$ if and only if $\lambda$ is eigenvalue of $A$ corresponding
to the eigenvector $\bm v=V(\lambda)\bm\alpha$.  Notice that, for
$p>q$, this system has always a solution since the matrix
$HV(\lambda)-\lambda V_{q\times p}(\lambda)$ has more columns than
rows so that $\hbox{\rm ker}(HV(\lambda)-\lambda V_{q\times
  p}(\lambda))\ne\{0\}$ and the multiplicity of $\lambda$ is given by
$p-\hbox{rank}(HV(\lambda)-\lambda V_{q\times p}(\lambda))$.

If $p=q$, equation \eqref{eq:sys} provides a balanced nonlinear
eigenvalue problem that we are going to analyze.

If $p<q$ and if the pair $(\lambda,\bm\alpha)$ solves \eqref{eq:sys},
then it solves also the balanced nonlinear eigenvalue problem
\begin{equation}\label{eq:sys1}
H_{p\times\infty}V(\lambda)\bm\alpha = \lambda V_p(\lambda)\bm\alpha,
\end{equation}
formed by the first $p$ equations of \eqref{eq:sys}.  Thus, we may
look for solutions $(\lambda,\bm\alpha)$ of \eqref{eq:sys1}, and, if
any, we may check if these are also solutions of \eqref{eq:sys}.

We may express the NEP \eqref{eq:sys} in a more convenient form by
using the Toeplitz structure of $T(a)$. This is the subject of the
next section.

\subsection{A different formulation}\label{sec:da}
Let $Z=(z_{i,j})$ be the shift matrix defined by $z_{i,i+1}=1$,
$z_{i,j}=0$ elsewhere. Then for any solution $\bm v$ of the linear
difference equation \eqref{eq:diffeq}, the shifted vector $Z^k \bm v$
is still a solution for any $k\ge 0$.  Moreover, if $\bm v\in\ell^2$
then also $Z^k\bm v\in\ell^2$, and if $\bm v^{(1)}$ and $\bm v^{(2)}$
are linearly independent, then also $Z^k\bm v^{(1)}$ and $Z^k\bm
v^{(2)}$ are linearly independent.  To show the latter implication,
assume that there exists a linear combination $\bm v=\alpha_1\bm
v^{(1)}+\alpha_2\bm v^{(2)}\ne 0$ such that $Z^k \bm v=0$. Then,
$v_i=0$ for $i\ge k+1$. But since $a_{-m}\ne 0$, we find that
$v_k=\ldots=v_1=0$, i.e., $\bm v=0$ that is a contradiction.

Therefore, if the columns of $V(\lambda)$ are a basis of the space of
the solutions in $\ell^2$, then also the columns of
$U(\lambda)=Z^mV(\lambda)$ form a basis of the same space. This
implies that the columns of $U(\lambda)$ are linear combinations of
the columns of $V(\lambda)$. That is, there exists a non singular
$p\times p$ matrix $S(\lambda)$ such that
$U(\lambda)=V(\lambda)S(\lambda)$ whence we have
$Z^mV(\lambda)=V(\lambda)S(\lambda)$.

If we multiply the rows from $m+1$ to $2m$ of the Toeplitz matrix
$T(a)-\lambda I$ by $V(\lambda)$ we get
\[
\begin{bmatrix}
a_{-m}&\ldots&a_{-1}&a_0-\lambda&a_1&\ldots&a_n\\
& \ddots&\ddots& \ddots&\ddots &\ddots&\ddots&\ddots \\
&& a_{-m}&\cdots&a_{-1}& a_0-\lambda&a_1&\cdots&a_n\\
\end{bmatrix}V(\lambda)=0.
\] 
Observing that $V(\lambda)=[V_{m\times p}(\lambda);Z^mV(\lambda)]$, we
may rewrite the identity as
\[
\begin{bmatrix}
B~&~(T(a)-\lambda I)_{m,\infty}
\end{bmatrix}\begin{bmatrix}
V_{m\times p}(\lambda)\\Z^m V(\lambda)
\end{bmatrix}=0,\quad B=\begin{bmatrix}
a_{-m}&\ldots&a_{-1}\\
&\ddots&\vdots\\
&&a_{-m}
\end{bmatrix}.
\]
Since $Z^mV(\lambda)=V(\lambda)S(\lambda)$, we get
\[
(T(a)-\lambda I)V(\lambda) = \begin{bmatrix}-B\\ 0_{\infty\times m}\end{bmatrix}
V_{m\times p}(\lambda)S(\lambda)^{-1}. 
\]
On the other hand, 
relying once again on the property $Z^mV(\lambda)=V(\lambda)S(\lambda)$, we find that
\[
\begin{split}(A-\lambda I)V(\lambda) &= -\begin{bmatrix}B
\\ 0_{\infty\times m}\end{bmatrix}V_{m\times p}(\lambda)S(\lambda)^{-1} +EV(\lambda) \\
&=\left(-  
\begin{bmatrix}
B &0_{m\times\infty} \\
0_{\infty\times m}&0_{\infty\times\infty}
\end{bmatrix}
V(\lambda) +EZ^mV(\lambda)\right)S(\lambda)^{-1}
\end{split}
\]
so that 
\begin{equation}\label{eq:yet}
(A-\lambda I)V(\lambda)=M V(\lambda)S(\lambda)^{-1},\quad M=  
\begin{bmatrix}
-B &0_{m\times\infty} \\
0_{\infty\times m}&0_{\infty\times\infty}
\end{bmatrix}
+EZ^m.
\end{equation}
The possibly nonzero rows of the matrix $M$ are the first
$q=\max(m,k_1)$ rows, which form the $q\times \infty$ matrix
$N=M_{q\times\infty}$, i.e., $M=[N;0]$. It is interesting to observe
that the matrix $EZ^m$ is obtained by shifting the columns of $E$ to
the right of $m$ places. This implies that the matrix $N$ takes one of
the following forms
\[
N=\begin{bmatrix}
-B&E_1\\0_{(q-m)\times m}&E_2
\end{bmatrix},\quad
N=\begin{bmatrix}
-B&E_1
\end{bmatrix},
\]
depending on whether $q>m$ or $q=m$, respectively, where
$E=[E_1;E_2]$, and $E_1$ has size $m\times \infty$ while $E_2$ has
size $(q-m)\times \infty$. In other words, the submatrices $E$ and $B$
do not overlap. This fact allows us to rewrite \eqref{eq:sys} as a set
of $q$ equations in $p$ unknowns in the more convenient
form \begin{equation}\label{eq:yetw} NV(\lambda)\bm\beta=0,\quad N=
\begin{bmatrix}
-B &E_1 \\
0_{(q-m)\times m}&E_2
\end{bmatrix}, \quad \bm\beta=S(\lambda)^{-1}\bm\alpha.
\end{equation}

Another observation is that multiplying equation \eqref{eq:yetw} on
the left by any invertible matrix provides an equivalent formulation
of the NEP\@. In particular, if $q>m$, consider the rank revealing QR
factorization $E_2=QR$ of the matrix $E_2$, assume that
$\hbox{rank}(E_2)=r_2$ and denote $\widetilde R$ the $r_2\times
\infty$ matrix formed by the first $r_2$ rows of $R$ so that
$R=[\widetilde R;0]$ and we may write $E_2=Q[\widetilde
  R;0_{(q-m-r_2)\times\infty}]$.

Multiplying \eqref{eq:yetw} to the left by $\hbox{diag}(I_m,Q^*)$, where $Q^*$ is the transposed Hermitian of $Q$,
yields
\begin{equation}\label{eq:yet1}
 WV(\lambda)\bm\beta=0,\quad 
W=\begin{bmatrix}
-B&E_1\\0_{r_2\times m}&\widetilde R
\end{bmatrix},\quad \bm v=V(\lambda)S(\lambda)\bm\beta=Z^mV(\lambda)\bm\beta.
\end{equation}

Observe that $W$ is a constant matrix of full rank, with
$m+\hbox{rank}(E_2)=m+r_2$ rows, while $V(\lambda)$ is a matrix
depending on $\lambda$. The eigenvalue problem for QT matrices is
reduced to the NEP \eqref{eq:yet1} which can take different forms
according to the way a basis of the solution of the difference
equation \eqref{eq:diffeq} is chosen.

We may conclude with the following result.

\begin{theorem}
  Let $\Omega$ be a connected component of the set $\mathbb C\setminus
  a(\mathbb T)$, let $\lambda\in\Omega$ and $p=m+\hbox{\rm
    wind}(a-\lambda)$. Let $V(\lambda)$ be a matrix whose $p$ columns
  form a basis of the space of solutions of the difference equation
  \eqref{eq:diffeq} belonging to $\ell^2$. If $p>q$ then all
  $\lambda\in\Omega$ are eigenvalues of $T(a)+E$. If $p\le q$ then
  $\lambda\in\Omega$ is eigenvalue of $A=T(a)+E$ corresponding to the
  eigenvector $\bm v\in\ell^2$ iff there exists $\bm\beta\in\mathbb
  C^p\setminus \{0\}$ which solves the nonlinear eigenvalue problem
  $WV(\lambda)\bm\beta=0$ of \eqref{eq:yet1}. In this case, $\bm
  v=Z^mV(\lambda)\bm\beta$.
\end{theorem}

\subsection{Choosing a basis: Vandermonde and Frobenius versions}\label{sec:basis}

Let the zeros $\xi_i$ of $a(z)-\lambda$ be simple and ordered as
\[
|\xi_1|\le\cdots\le |\xi_p|<1\le|\xi_{p+1}|\le\cdots\le |\xi_{m+n}|,
\]
and let $V(\lambda)=(\xi_j^{i-1})_{i\in\mathbb Z^+,j=1,\ldots,p}$ be
the $\infty\times p$ Vandermonde matrix associated with
$\xi_1,\ldots,\xi_p$.  The columns $\bm v^{(1)},\ldots, \bm v^{(p)}$
of $V(\lambda)$ provide a basis of the set of solutions of the
difference equation \eqref{eq:diffeq} that belong to $\ell^2$, so that
$\bm v$ is an eigenvector of $A$ corresponding to $\lambda$ if and only
if there exists $\bm\alpha=(\alpha_i)\in\mathbb C^p\setminus \{0\}$
such that $\bm v=\sum_{j=1}^p \alpha_j \bm v^{(j)}$ and \eqref{eq:sys}
is satisfied.  The same argument can be applied in the
case of confluent zeros considering a generalized Vandermonde matrix.

The formulation \eqref{eq:yet1} where $V(\lambda)$ is the
(generalized) Vandermonde matrix associated with the roots $\xi_i$ of
$a(z)-\lambda$ is referred to as the {\em Vandermonde version} of the
problem.
It is well known that the zeros of a polynomial are severely
ill-conditioned if they are clustered. This may make the choice of the
basis $\bm v^{(i)}$, given by the columns of the Vandermonde matrix,
unsuited in some problems.
A way to overcome this issue is to consider
the {\em Frobenius version} of the NEP obtained in the following way.

For the sake of notational simplicity, in the following we write $V$
in place of $V(\lambda)$.  For simple roots, write the Vandermonde
matrix $V$ in the form $V=[V_p;V_pD^p;V_pD^{2p};\ldots]$, with
$D=\hbox{diag}(\xi_1,\ldots,\xi_p)$, and define $U:=VV_p^{-1}$.
Recall that $V_pD^pV_p^{-1}=F^p$, where
$F=Z_p-\bm e_p[s_0,s_1,\ldots,s_{p-1}]$ denotes the companion (Frobenius)
matrix associated with the polynomial
$s(z)=(z-\xi_1)\cdots(z-\xi_p)=\sum_{i=0}^{p-1}s_iz^i+z^p$, see for
instance \cite{blm:book}. Here, $\bm e_p=[0,\ldots,0,1]^T\in\mathbb
R^p$. For multiple roots, a similar construction can be made with the
generalized Vandermonde matrix and where $D$ is a block diagonal
matrix whose diagonal blocks are Jordan blocks associated with the
distinct roots of $a(z)-\lambda$ having modulus smaller than~$1$.
 
 Denote $G:=F^p$ so that the columns of $U=VV_p^{-1}=[I;G;G^2;\ldots]$
 provide a different basis of the set of solutions of the linear
 difference equation \eqref{eq:diffeq}. The NEP \eqref{eq:yet1} can be
 equivalently rewritten as
\begin{equation}\label{eq:eigF0}
WU\bm\gamma=0,\quad \bm v=Z^m U\bm\gamma,
\end{equation}

We refer to \eqref{eq:eigF0} as the {\em Frobenius version} of the
problem.  Observe that in the Frobenius form, it is not relevant if
the roots of $a(z)-\lambda$ are multiple or numerically clustered, in
fact the matrix $G=F^p$ exists and can be computed independently of
the location of the roots of $s(z)$.

Notice that if $m+r_2=p$, then the matrix $W$ can be partitioned into
$p\times p$ blocks as $W=[W_0,W_1,W_2,\ldots]$ and $ W U$ can be
rewritten in terms of a matrix power series as $ W U=\sum_{i=0}^\infty
W_iG^{i}$. The following result provides information in this regard
\cite[Chapter 3]{blm:book}.

\begin{theorem}\label{th:mateq}
  Assume that $a(z)=\sum_{i=-m}^n a_iz^i$, where $a_{-m},a_n\ne 0$,
  has roots $\xi_i$,
  $i=1,\ldots,m+n$ such that $|\xi_1|\le\cdots\le|\xi_p|<1\le
  |\xi_{p+1}|\le\ldots\le|\xi_{m+n}|$ and denote
  $s(z)=\prod_{i=1}^p(z-\xi_i)$.  Define
  $A_k=(a_{j-i+kp-m+p})_{i,j=1,p}$ for $k=-1,0,1,\ldots$ where we
  assume $a_{\ell}=0$ if $\ell<-m$ or $\ell>n$. Let $F$ be the Frobenius matrix
  associated with the factor $s(z)$. Then
  $G=F^p$ is the unique solution  of the
  matrix equation
\begin{equation}\label{eq:mateq}
\sum_{k=-1}^\infty A_kX^{k+1}=0 ,
\end{equation}
having minimum spectral radius $\rho(G)$, moreover,
 $\rho(G)=|\xi_p|$.
\end{theorem}

Notice that the blocks $A_k$ defined in the above theorem are obtained
by partitioning the Toeplitz matrix $T\bigl(z^{m-p}a(z)\bigr)$ into
$p\times p$ blocks which are themselves Toeplitz.  Moreover, since
$T\bigl(z^{m-p}a(z)\bigr)$ is a banded matrix, then $A_k=0$ for $k$ sufficiently large. In the literature,
there are several effective algorithms for the numerical computation
of $G$, based on fixed point iterations or on doubling techniques. We
refer the reader to \cite{bfgm}, \cite{blm:book}, \cite{BH13}, and
\cite{BH14}, for more details.

\section{The numerical algorithms}
In this section we describe our algorithms to refine a given
approximation of an eigenvalue $\lambda$ of $A=T(a)+E$, while
in Section \ref{sec:init}
 we will discuss how to get the initial approximation. The
algorithms require: a function $g(x):\mathbb C\to\mathbb C$ such that
the fixed point iteration $\lambda_{\nu+1}=g(\lambda_\nu)$ converges locally to
the eigenvalue $\lambda$, solution of the NEP \eqref{eq:yet1}, and a
choice of the basis $V(\lambda)$ of the solutions of \eqref{eq:diffeq}
belonging to $\ell^2$.

The general scheme is reported in the Template Algorithm
\ref{alg:template}. This algorithm, for an initial approximation
$\lambda_0\in \mathcal U_w:=\{\lambda \in \mathbb C\setminus a(\mathbb
T)\,:\, \mbox{wind}(a-\lambda)=w\}$ of the eigenvalue, provides either
a more accurate approximation to the corresponding eigenpair, or a
message with the following possible cases: 1) all the elements in the set
$\mathcal U_w$ are eigenvalues; 2) the generated sequence exited from
$\mathcal U_w$; 3) it holds $p<q$ and the approximated solution solves the
first $p$ equations but not the full NEP \eqref{eq:yet1}; 4)
convergence did not occur after the maximum number of allowed
iterations.

\begin{algorithm}
\caption{Template Algorithm}
\label{alg:template}
\begin{algorithmic}[1] 
    \REQUIRE the coefficients of $a(z)=\sum_{i=-m}^n a_iz^i$ and the
    correction matrix $E$; an error bound $\epsilon>0$; an initial
    approximation $\lambda_0\in\mathbb C$; an upper bound {\tt maxit} to the
    number of iterations; a function $g(x):\mathbb C\to\mathbb C$
    defining a fixed point iteration to solve the NEP \eqref{eq:yet1};
    a rule to generate $V(\lambda)$.
 
    \ENSURE An approximation $\mu$ to an eigenvalue $\lambda$, the
    vector $\bm \beta$ providing an approximation to the corresponding
    eigenvector according to \eqref{eq:yet1}, together with a message.
    \STATE Construct the matrix $W$ of \eqref{eq:yet1} together with
    the scalar $r_2$; compute $w_0=\hbox{wind}(a-\lambda_0)$, $p_0=m+w_0$
    and $q=m+r_2$. Set $\nu=0$.
    \WHILE{$\nu<${\tt
        maxit}} \STATE\label{step3} Compute
    $w=\hbox{wind}(a-\lambda_\nu)$, $p=m+w$.  \STATE if $w\ne w_0$
    then output {\tt `out of $\mathcal U_{w_0}$'} and stop; otherwise
    set $w_0=w$.  \STATE if $p>q$ then output {\tt `continuous set of
      eigenvalues'} and stop.  \STATE if $p=q$ then perform one step
    of the fixed point iteration $\lambda_{\nu+1}=g(\lambda_\nu)$; set
    $\nu=\nu+1$, compute $\bm\beta$ and $\bm v$ according to
    \eqref{eq:yet1}, compute the residual error ${\tt
      res}=\|((A-\lambda_\nu I)\bm v)_{q\times 1}\|/\|\bm v_{q\times
      1}\|$; if ${\tt res}\le\epsilon$, output {\tt `isolated
      eigenvalue (p=q)'} together with $\mu=\lambda_\nu$, $\bm\beta$
    and exit, otherwise continue from step \ref{step3};

\STATE if  $p<q$ then perform one step of  the
fixed point iteration $\lambda_{\nu+1}=g(\lambda_\nu)$ applied to 
\eqref{eq:yet1} restricted to the first $p$ components.
Check if the residual error in the first $p$ components is less than $\epsilon$. If not,  continue from step \ref{step3}, otherwise check if rank$( WV(\lambda_\nu))$ is less than $p$. If so, output {\tt `isolated eigenvalue (p<q)'} together with $\mu=\lambda_\nu$ and $\bm\beta$, and exit; otherwise output {\tt `non converging sequence (p<q)'} and exit;
\STATE  Stop if the maximum number of iterations {\tt maxit} has been reached; in this case, output  {\tt `Maximum number of iterations exceeded'}.
     \ENDWHILE

\end{algorithmic}
\end{algorithm}

Now, we deal with algorithmic issues encountered in the design of the
fixed point iterations to solve the nonlinear eigenvalue problem
\eqref{eq:yet1}. This analysis is needed to design algorithms to
implement the function $g(x)$ used in the Template Algorithm
1. Without loss of generality, we assume that the nonlinear eigenvalue
problem is balanced. This case is encountered if $p=q$ or if $p<q$
where we consider the subset of the first $p$ equations in
\eqref{eq:yet1}.

We essentially analyze Newton's iteration applied to the determinantal versions of the problem, that is,  $\det( WV)=0$, $\det(WU)=0$, in the Vandermonde and in the Frobenius forms, respectively. 
Before doing that, we discuss on how to compute the winding number of $a(z)-\lambda$, since this is a fundamental step in the design of the overall algorithm.

\subsection{Computing the winding number}
The winding number $w$ of the Laurent polynomial $a(z)-\lambda$ can be
computed in different ways. The most elementary one is to express $w$
as $w=p-m$, where $p$ is the number of zeros of $a(z)-\lambda$ of
modulus less than~1. Any root-finding algorithm applied to the polynomial $z^m(a(z)-\lambda)$ can be used for this
purpose, for instance, the command {\tt roots} of Matlab provides
approximations to all the roots of $z^m(a(z)-\lambda)$, and we may count
how many roots have modulus less than~1. This approach has the
drawback that polynomial roots are ill-conditioned when clustered, so
that we may encounter instability if there are clusters of roots of
modulus close to 1.

A second approach is based on equation \eqref{eq:wind} that expresses
$w$ as ratio of two integrals. The integrals can be approximated by
the trapezoid rule at the Fourier points using two FFTs. In this case,
the presence of roots of the polynomial close to the unit circle may
lead to a large number of Fourier points with a consequent slow down
of the CPU time.

A third approach, that is the one we have implemented, relies on
Graeffe's iteration \cite{graeffe}, that is based in the following
observations. Given a polynomial $b(z)$ of degree $m+n$, the
polynomial $c(z)=b(z)b(-z)$ is formed by monomials of even degree,
i.e., there exists a polynomial $b_1(z)$ of degree $m+n$ such that
$b_1(z^2)=c(z)$. Therefore, the roots of $b_1(z)$ are the square of
the roots of $b(z)$.  Consider the sequence defined by the Graeffe
iteration $b_{k+1}(z^2)=b_k(z)b_k(-z)$ with initial value
$b_0(z)=b(z)$. It turns out that the winding number of $b_k(z)$ is
constant. Moreover, if $b(z)$ has $m$ zeros of modulus less than 1 and
$n$ zeros of modulus greater than 1, then the limit for $k\to\infty$
of $b_k(z)/\theta_k$ is $z^m$. Here $\theta_k$ is the
coefficient of maximum modulus of $b_k(z)$.  This means that there
exists an index $k$ such that the coefficient of $z^m$ in $b_k(z)$ has
modulus greater than $\frac12\|b_k(z)\|_1$, where $\|b_k(z)\|_1$ is the
sum of the moduli of all the coefficients of $b_k(z)$. In view of
Rouch\'e theorem, the latter inequality is a sufficient condition to
ensure that $b_k(z)$ has $m$ roots of modulus less than 1.

Indeed, if there are zeros of modulus 1 then this procedure might not terminate. Therefore, if the number of Graeffe iterations exceeds a given upper bound,
then the explicit computation of the polynomial roots is performed.
These arguments support Algorithm \ref{alg:wind}
for counting the number of roots of a polynomial of modulus less than 1.

\begin{algorithm}
\caption{Count roots}
\label{alg:wind}
\begin{algorithmic}[1]
    \REQUIRE The coefficients of a polynomial $b(z)=\sum_{i=0}^{d} b_iz^i$  of degree $d$;  an upper bound {\tt maxit} to the number of iterations.
    
    \ENSURE Either the number of zeros of $b(z)$ of modulus less than 1, or the message {\tt `failure'}.
    \STATE Set $b_0(z)=b(z)$, $\nu=0$.
    \WHILE{$\nu<${\tt maxit}}
    \STATE Compute the coefficients of $c(z)=\sum_{i=0}^dc_iz^i$ such that 
    $c(z^2)=b_\nu(z)b_\nu(-z)$, let $h$ be the minimum index such that
    $|c_h|=\max_i|c_i|$, and set  $b_{\nu+1}(z)=c(z)/c_h$.
    \STATE $\nu=\nu+1$
    \STATE If $\|b_\nu(z)\|_1<2$ then output $h$ and exit
    \ENDWHILE
    \STATE Compute the zeros of $c(z)$ and output the number of zeros
    of modulus less than 1 together with the warning: {\tt `Reached
      the maximum number of iterations'}
    \end{algorithmic}
    \end{algorithm}
 
\subsection{Implementing Newton's iteration}\label{sec:newt}
In this section we analyze the computational issues concerning the
implementation of Newton's iteration applied either to
$f_V(\lambda)=\det \Phi_V(\lambda)$, where $\Phi_V(\lambda)=WV(\lambda)$
in the Vandermonde approach, or to $f_F(\lambda)=\det \Phi_F(\lambda)$,
where $\Phi_F(\lambda)=WU(\lambda)$ in the Frobenius approach. We use
the symbol $\Phi(\lambda)$ to denote either $\Phi_V(\lambda)$ or
$\Phi_F(\lambda)$, similarly we do for $f(\lambda)$.  In all cases,
$\Phi(\lambda)$ is assumed to be a $p\times p$ matrix. This is true if
$q=p$, and also in the case where $ q>p$ when we consider only the
first $p$ rows of $ WV(\lambda)$ or of $WU(\lambda)$.

Since $U(\lambda)=V(\lambda)V_p^{-1}$ then we have
$\Phi_V(\lambda)=\Phi_F(\lambda)V_p$ so that
$f_V(\lambda)=f_F(\lambda)\det V_p(\lambda)$.  We recall that if the
function $f(\lambda)$ has continuous second derivative, then Newton's
method applied to the equation $f(\lambda)=0$, given by
$z_{\nu+1}=z_\nu-f(\lambda_\nu)/f'(\lambda_\nu)$, locally converges to
a zero of $f(\lambda)$. The convergence is at least quadratic if the
zero is simple, it is linear if the zero is multiple.  If
$\Phi(\lambda)$ has entries
with continuous second derivative, then also
$f(\lambda)=\det \Phi(\lambda)$ has
continuous second derivative and for the Newton's correction
$f(\lambda)/f'(\lambda)$ we have
\begin{equation}\label{eq:nc}
f(\lambda)/f'(\lambda)=1/\hbox{trace}(\Phi(\lambda)^{-1}\Phi'(\lambda)).
\end{equation}
A simple calculation shows that
if $\Phi(\lambda)=P(\lambda)Q(\lambda)$ then 
\begin{equation}\label{eq:nc1}
f(\lambda)/f'(\lambda)=1/\left(\hbox{trace}(P(\lambda)^{-1}P'(\lambda))+\hbox{trace}(Q(\lambda)^{-1}Q'(\lambda))\right).
\end{equation}
In particular, since $U(\lambda)=V(\lambda)V_p$, assuming $f_F(\lambda)$ and $f_V(\lambda)$ differentiable, we have
\[
f_F(\lambda)/f'_F(\lambda)=f_V(\lambda)/f'_V(\lambda)+\hbox{trace}(V_p(\lambda)^{-1}V'_p(\lambda)).
\]

\subsubsection{Vandermonde version}
In order to apply Newton's iteration in the Vandermonde version, we
have to assume that the roots $\xi_i(\lambda)$ of the polynomial
$a(z)-\lambda$ have continuous second derivative.  It is well known
that if the coefficients of a polynomial $p_\lambda(z)$ of degree
$\nu$ are analytic functions of $\lambda$, and if for a given
$\lambda_0$ the polynomial has simple roots $\xi_1,\ldots,\xi_\nu$,
then for $\lambda$ in a neighbourhood of $\lambda_0$, there exist
$\xi_1(\lambda),\ldots,\xi_\nu(\lambda)$ analytic functions that are
roots of $p_\lambda(z)$ and $\xi_i(\lambda_0)=\xi_i$, for
$i=1,\ldots,\nu$.  Indeed, the polynomial $z^m(a(z)-\lambda)$ has
coefficients that are analytic for $\lambda\in\mathbb C$, therefore
$\xi_i(\lambda)$ are analytic functions as long as the zeros remain
simple. In this subsection we assume this condition.

In order to compute the Newton correction by means of \eqref{eq:nc} we
need to compute the entries of the Vandermonde matrix
$V(\lambda)$. Therefore we assume we are given a polynomial rootfinder
which approximates the roots of $z^m(a(z)-\lambda) $ so that we may
select the $p$ roots of modulus less than 1. For this task we rely on
the Matlab command {\tt `roots'}. Then we need to compute
$V'(\lambda)$, i.e., the derivative of the entries of $V(\lambda)$.
Concerning this task we have $(\xi_j^ {i})'=i\xi_j^{i-1}
\xi'_j$. Moreover, since $a(\xi_j)-\lambda=0$, taking the derivative
of this equation yields $a'(\xi_j)\xi'_j-1=0$, whence
$\xi'_j=1/a'(\xi_j)$.  Therefore, we are able to implement the Newton
iteration where the Newton correction takes the form \eqref{eq:nc}
with $ (V'(\lambda))_{i,j}=(i-1)\xi_j^{i-2}/a'(\xi_j)$.

\subsubsection{Frobenius version}
Consider the case $\Phi(\lambda)= WU(\lambda)$, where $U(\lambda)$ is
the matrix defined in Section~\ref{sec:basis}. In order to evaluate
the Newton correction, we have to compute the matrix $G$ of minimal
spectral radius which solves the matrix equation \eqref{eq:mateq},
then evaluate, the powers $G^j$ and their derivatives $(G^j)'$, for
$j\ge 0$.

Firstly, we discuss on how to compute $G$. This matrix can be obtained
by the coefficients of the polynomial $s(z)$ collecting the zeros of
$a(z)-\lambda$ of modulus less than 1, that yields the Frobenius
matrix $F$ and in turn $G=F^p$. In our implementation we compute
directly the matrix $G$ as the solution of minimal spectral radius of
equation \eqref{eq:mateq} (compare Theorem \ref{th:mateq}). For this
task, the algorithm of Cyclic Reduction, having a quadratic
convergence, can be effectively applied~\cite{bfgm}.
It is worth pointing out that the
first row of $-G$ contains the coefficients $s_0,\ldots,s_{p-1}$ of
the sought monic factor $s(z)$, so that these coefficients are known
once the matrix $G$ has been computed.

Secondly, we show how to compute the derivative of the coefficients
$s_0,\ldots,s_{p-1}$ of $s(z)$ with respect to $\lambda$.  The
polynomial $z^m (a(z)-\lambda)$ can be factorized as
$z^m(a(z)-\lambda)=s(z)u(z)$, where $u(z)$ has  zeros of
modulus greater than or equal to 1, and $s(z)$ has zeros of modulus less than 1.
Therefore, setting
$\widehat p=m+n-p$, we have the equation
\begin{equation}\label{eq:F}
\small
\begin{bmatrix}
a_{-m}\\\vdots\\a_0-\lambda\\\vdots\\a_n
\end{bmatrix}
=\begin{bmatrix}
u_0\\
u_1&u_0\\
\vdots&\ddots&\ddots\\
u_{\widehat p}&\ddots&\ddots&\ddots\\
         &\ddots&\ddots&\ddots&u_0\\
         &      &\ddots&\ddots&u_1\\
         &      &      &\ddots&  \vdots \\
         &      &      &      &u_{\widehat p}
\end{bmatrix}
\begin{bmatrix}s_0\\s_1\\\vdots\\s_p
\end{bmatrix}
=\begin{bmatrix}
s_0\\
s_1&s_0\\
\vdots&\ddots&\ddots\\
s_p&\ddots&\ddots&\ddots\\
         &\ddots&\ddots&\ddots&s_0\\
         &      &\ddots&\ddots&s_1\\
         &      &      &\ddots&  \vdots \\
         &      &      &      &s_p
\end{bmatrix}
\begin{bmatrix}u_0\\u_1\\\vdots\\u_{\widehat p}
\end{bmatrix}.
\end{equation}
Denote by $U$ and $S$ the two matrices in the above equation and
observe that they have size $(m+n+1)\times (p+1)$ and $(m+n+1)\times
(\widehat p+1)$. Since $s_p=1$ and $u_{\widehat p}=a_n$, then
$s'_p=u'_{\widehat p}=0$.
Set $\bm u=[u_1,\ldots,u_{\widehat p}]^T$, $\bm s=[s_0,\ldots,s_p]^T$.
Taking derivatives with respect to
$\lambda$, and denoting $\bm e_{m+1}$ the vector with null components
except the $(m+1)$-st which is 1, yields the system $ -
\bm e_{m+1}=U\bm s'+S\bm u' $ which can be rewritten as
\begin{equation}\label{eq:sylv}
[\widehat U,\widehat S]
\begin{bmatrix}
\widehat {\bm s}'\\ \widehat {\bm u}'
\end{bmatrix}
=-\bm e_{m+1},
\end{equation}
where $\widehat {\bm s}=[s_0,\ldots,s_{p-1}]^T$, $\widehat
{\bm u}=[u_0,\ldots,u_{\widehat p-1}]^T$, and $\widehat U$ and $\widehat S$
are the matrices obtained from $U$ and $S$, respectively, by removing
the last column and the last row. This is a system formed by $m+n$
equations and $m+n$ unknowns. Moreover, the matrix $[\widehat
  U,\widehat S]$ is invertible since it is a resultant matrix
associated with polynomials having no zeros in common.  Therefore we
have
\[
\begin{bmatrix}
\widehat {\bm s}'\\ \widehat {\bm u}'
\end{bmatrix}=-[\widehat U,\widehat S]^{-1}\bm e_{m+1}.
\]

Thirdly, we explain how to compute $G'$ using the derivatives of
$s_0,\ldots,s_{p-1}$. We rely on the Barnett factorization
\cite{barnett} that provides an LU factorization of the matrix $G=F^p$
\begin{equation}\label{eq:barnett}
\small 
F^p= -\mathcal L^{-1}\mathcal U,\quad
\mathcal L=\begin{bmatrix}
s_p\\ s_{p-1}&s_p\\
\vdots&\ddots&\ddots\\
s_1&\ldots&s_{p-1}&s_p
\end{bmatrix},\quad \mathcal U=\begin{bmatrix}
s_0&s_1&\ldots&s_{p-1}\\
&s_0&\ddots&\vdots\\
&&\ddots&s_1\\
&&&s_0
\end{bmatrix},
\end{equation}
where $\mathcal L$ and $\mathcal U$ are lower triangular and upper
triangular Toeplitz matrices, respectively.  Applying the Barnett
factorization \eqref{eq:barnett} to our problem, we have
\begin{equation}\label{eq:barFp}
(F^p)'=-\mathcal L^{-1}\mathcal U'+\mathcal L^{-1}\mathcal L'\mathcal L^{-1}\mathcal U,
\end{equation}
 where $\mathcal L'$ and $\mathcal U'$ are the derivatives of
 $\mathcal L$ and $\mathcal U$, respectively, that are determined by
 the derivative $s_j'$, $j=0,1,\ldots,p$. We may observe that the cost
 of computing $(F^p)'$ by means of \eqref{eq:barFp} amounts to
 $O(p^3)$ arithmetic operations which, due to the triangular Toeplitz
 structure and to the fast algorithms for triangular Toeplitz matrix
 inversion and for Toeplitz-vector multiplication can be lowered to
 $O(p^2\log p)$.

Finally, we discuss on how to compute $G^j$ and $(G^j)'$ given $G'$.
{}From the relation $G^j=G^{j-1}G$ we obtain
$(G^j)'=(G^{j-1})'G+G^{j-1}G'$. This expression allows us to compute
$(G^i)'$ and $G^i$ for $i=1,\ldots,k$ according to the following
equations
\[
\begin{array}{l}
G^i = G G^{i-1}\\
(G^i)'=(G^{i-1})'G+G^{i-1}G'
\end{array}\quad  i=2,\ldots,k.\\
\]
Clearly, the cost of this computation is $2(k-1)$ matrix
multiplications and $k-1$ matrix additions, for an overall cost of
$2(k-1)p^3+O(kp^2)$ arithmetic operations.

In our implementation, we have adopted the algorithm based on the
Barnett factorization for its simplicity, but other effective
techniques can be used. For instance, a different approach is based on
the structure of $F$ and on the fact that $F'=-\bm e_p \bm s'^T$.
  Indeed, $(F^k)'$ is such that $ (F^k)' =
(F^{k-1})' F+F^{k-1}F',~ F'=-\bm e_p{\bm s'}^T $.  This implies that
\[
(F^k)' = (F^{k-1})' F +\bm f_{k-1}{\bm s'}^T,\quad \bm f_k=F\bm f_{k-1},
\quad k=1,\ldots,p,
\]
where $\bm f_0=-\bm e_p$. A careful computational analysis shows that this
computation can be performed in $O(p^2)$ arithmetic operations.

A slightly different approach can be carried out as follows.
Recall that the last row of $F$ is $-\bm s^T$  and that $F'=-\bm e_p\bm s'^T$.
Given $\bm s$ and $\bm s'$,  write
\[
(F^p)'=-\sum_{\ell=0}^{p-1} F^{p-1-\ell}\bm e_p \bm s'^T F^\ell=-\sum_{\ell=0}^{p-1}
\bm v_{p-1-\ell}\bm s'^T F^\ell,
\]
where $\bm v_\ell\!=\!F^\ell \bm e_p$. Observe that the vector $\bm v_\ell$ is such that
$
\bm v_\ell\!=\!\begin{bmatrix} 0& \ldots & 0 & \sigma_p&\ldots& \sigma_{p-\ell}\end{bmatrix}^T,
$
with $\sigma_p\!=\!1$ and $\sigma_{p-\ell}\!=\!-\sum_{h=1}^{\ell} s_{p-h}\sigma_{p-\ell+h}$, for $\ell=1,\ldots,p-1$. For the rows $\bm r_1^T,\ldots,\bm r_p^T$ of $(F^p)'$ we have
\[
\begin{aligned}
\bm r_\ell^T&=\sigma_{p-\ell+1}\bm s'^T+\bm r_{\ell-1}^TF
\\&=\sigma_{p-\ell+1}\bm s'^T+(F^p)'_{\ell-1,p}\bm s^T+[0\ (F^p)'_{\ell-1,1}\ \cdots (F^p)'_{\ell-1,p-1}],\quad \ell=2,\ldots,p
\end{aligned}
\]
and $\bm r_1^T=\sigma_p\bm s'^T=\bm s'^T$.
The cost of the procedure is given by the computation of $\sigma_1,\ldots,\sigma_p$ that requires $p^2-p$ operations, and the recursion for $\bm r_1^T,\ldots,\bm r_p^T$ that requires about $4p^2$ operations.

\subsubsection{Convergence of Newton's iteration}
We have seen that in the Vandermonde formulation, the function
$f_V(\lambda)$ is holomorphic in $\mathbb C\setminus a(\mathbb T)$
as long as the roots of the Laurent polynomial $a(z)-\lambda$ are simple.
Here we prove that the function $f_F(\lambda)$ is holomorphic in
$\mathbb C\setminus a(\mathbb T)$ under no additional condition. We
rely on the implicit function theorem for functions of complex
variable given in the following form \cite[Theorem 15]{dangelo}.
\begin{theorem}\label{th:implicit}
  Let $F:\mathcal V\subset \mathbb C^{k}\times \mathbb C^{q}\to\mathbb C^q$ be a
  holomorphic mapping such that the linear mapping $\frac{\partial
    F}{\partial w}(z_0,w_0):\mathbb C^q\to\mathbb C^q$ is invertible,
  where $(z_0,w_0)\in\mathcal V$. Then there are neighborhoods
  $\mathcal U$ and $\mathcal A$, $(z_0,w_0)\in\mathcal U$,
  $z_0\in\mathcal A$, and a holomorphic mapping $g:\mathcal
  A\to\mathbb C^q$, such that $F(z,w)=F(z_0,w_0)$ if and only if
  $w=g(z)$ for $(z,w)\in\mathcal U$.
\end{theorem}

Observe that for $\lambda\in\Omega$ the winding number of
$a(z)-\lambda$ is constant, where $\Omega$ is a connected component of
$\mathbb C \setminus a(\mathbb T)$. Therefore, the polynomial
$z^m(a(z)-\lambda)$ has $p=m+w$ roots of modulus less than 1 and
$\widehat p=m+n-p$ roots of modulus greater than 1. Thus, there exists
the Wiener-Hopf factorization $z^m(a(z)-\lambda)=s(z)u(z)$, where $s(z)$ is
the monic polynomial of degree $p$, with coefficients $s_i$,
$i=0,\ldots,p$, having roots of modulus less than 1, while $u(z)$, of
degree $\widehat p$ and coefficients $u_i$, $i=0,\ldots,\widehat p$,
has roots of modulus greater than 1.  Consider the function
$F(\lambda;s_0,\ldots,s_{p-1},u_0,\ldots,u_{\widehat p})=Us-\widehat
a=Su-\widehat a$, where $s=(s_0,\ldots,s_{p-1},1)^T$,
$u=(u_0,\ldots,u_{\widehat p})^T$, $\widehat
a=(a_{-m},\ldots,a_{-1},a_0-\lambda,a_1,\ldots, a_n)^T$, and where the
matrices $U$ and $S$ are defined in \eqref{eq:F}.  The function $F$ is
defined in $\mathbb C\times\mathbb C^{m+n+1}$ and takes values in
$\mathbb C^{m+n+1}$. A direct computation shows that the matrix of
partial derivatives of $F$ with respect to $s_i$ and to $u_j$ is given
by $[\widetilde U,S]$, where $\widetilde U$ is the matrix obtained by
removing the last column of $U$. This matrix is invertible since its
last row is $[0,\ldots,0,1]$ and the leading principal submatrix of
size $m+n$ coincides with $[\widehat U,\widehat S]$ in equation
\eqref{eq:sylv} that is invertible.

Therefore, we may apply Theorem \ref{th:implicit} to the function $F$
with $k=1$, $q=m+n+1$, where $F(z_0,w_0)=0$, and conclude with the
following result.
\begin{theorem}
  Let $\Omega$ be any connected component of $\mathbb C\setminus
  a(\mathbb T)$. Then, for $\lambda\in\Omega$ the function
  $f_F(\lambda)=\det \Phi_F(\lambda)$ is holomorphic.
\end{theorem}

\section{Choosing the initial approximation}\label{sec:init}
The algorithms presented in the previous sections can be used for
refining a given approximation to an isolated eigenvalue of a QT
matrix $A$, once an initial approximation is available. In this
section, we investigate the problem of determining initial
approximations to each isolated eigenvalue of $A$. More specifically,
we show that, if $A$ is Hermitian then for any isolated eigenvalue
$\lambda$ of $A$, and for any $\epsilon>0$ there exists an integer $N$
and an eigenvalue $\mu$ of the $N\times N$ leading principal submatrix $A_N$ 
(finite section) of $A$
such that $|\lambda-\mu|\le\epsilon$. That is, for each isolated
eigenvalue $\lambda$ of $A$ we may find a sufficiently close
approximation to $\lambda$ among the eigenvalues of the $N\times N$
matrix $A_N$ for a sufficiently large value of $N$.

For non-Hermitian matrices we have a weaker result: we
show that for any eigenvalue $\lambda$ of $A$ and for any positive
$\epsilon$, there exists $N_0>0$ such that for any $N\ge N_0$
$\lambda$ belongs to  the
$\epsilon$-pseudospectrum $\hbox{sp}_\epsilon$ of $A_N$  defined as
$\hbox{sp}_\epsilon(A_N)=\{z\in\mathbb C: \quad \|(A_N-z I)^{-1}\|\ge
\epsilon^{-1}\}$.

This fact enables us to implement a heuristic approach that, given
$A$, selects a sufficiently large value of $N$, computes all the
eigenvalues of $A_N$ and applies to each eigenvalue of $A_N$ one of
the fixed-point methods described in the previous section, and finally
selects the values for which the numerical convergence occurs.

Since we do not have an explicit formal relation between $\epsilon$
and $N$, and since we do not have a theoretical bound to the radius of
the convergence neighborhood of Newton's iteration, this strategy
remains a heuristics approach. Nevertheless, from our implementation
and from the experiments that we performed, this strategy turns out to
be practically effective.

\subsection{The case of Hermitian matrices}
If $A$ is Hermitian then the Bauer-Fike theorem provides a helpful
tool to show that the isolated eigenvalues of $A$ can be approximated
by the eigenvalues of $A_N$.

Let $A=T(a)+E$ be a QT matrix, $a(z)=\sum_{j=-m}^na_jz^j$, $E$ compact
correction with support $h_1\times h_2$, i.e., its entries outside the
leading $h_1\times h_2$ submatrix are zero.  Let $A_N$ be the $N\times
N$ leading principal submatrix of $A$.  Let $A\bm v=\lambda\bm v$ be
such that $\lambda$ is an isolated eigenvalue of $A$ and $\bm v=(v_i)$
has exponential decay, i.e., $\lim_j |v_j|^\frac1j = \xi$ for
$0<\xi<1$, and $\sum_j|v_j|^2=1$. Denote $Y\in\mathbb C^{n\times n}$
the lower triangular Toeplitz matrix with first column
$(a_n,\ldots,a_1)^T$. Due to the exponential decay of $v_i$, for any
$\epsilon>0$ there exists $N_0>0$ such that for any $N\ge N_0$ it
holds that $\|Y \bm w_N\|\le \|Y\|\,\|\bm w_N\|\epsilon$, where $\bm
w_N=(v_{N+1},\ldots,v_{N+n})^T$.

If $N>\max(m,n,h_1,h_2, N_0)$ set $\bm v_N=(v_1,\ldots,v_N)^T$,
$\bm u_N=[0_{N-n}; Y\bm w_N]$
 and rewrite
 the condition  $A\bm v=\lambda\bm v$ as 
\begin{equation}\label{eq:bf}
A_N \bm v_N+\bm u_N=\lambda \bm v_N.
\end{equation}
Defining $C_N=\frac 1{\bm v_N^*\bm v_N} \bm u_N\bm v_N^*$, we may
rewrite \eqref{eq:bf} as $ (A_N+C_N)\bm v_N=\lambda\bm v_N $.  That
is, $\lambda$ is eigenvalue of an $N\times N$ matrix which differs
from $A_N$ by the correction $C_N$. Observe also that the matrix $C_N$ satisfies the inequality
$\|C_N\|\le\frac1{\|\bm v_N\|}\|Y\|\cdot\|\bm w_N\|\le\frac1{\|\bm
  v_N\|}\epsilon$.

That is, we may look at an isolated eigenvalue $\lambda$ of $A$ as an
eigenvalue of a finite matrix obtained by perturbing the finite matrix
$A_N$. Therefore we may invoke the classical perturbation theorems for eigenvalues  of
finite matrices.  For instance we can apply the Bauer-Fike theorem.

\begin{theorem}[Bauer-Fike]
    Let $A$ be a diagonalizable matrix, i.e., there exists $S$ such
    that $S^{-1}AS=D$, $D$ diagonal, and let $\|\cdot\|$ be an
    absolute norm. Then, for any eigenvalue $\lambda$ of $A+C$ there
    exists an eigenvalue $\mu$ of $A$ such that $|\lambda-\mu|\le\|C\|
    \cdot \|S\|\cdot \|S^{-1}\|$.
\end{theorem}

Observe that the $p$-norms are absolute, i.e., $\|\bm v\|=\|(|v_i|)\|$
for any $\bm v=(v_i)$.

Therefore, if $A$ is Hermitian, then $A_N$ is Hermitian and
consequently $S$ can be chosen to be unitary so that for the 2-norm we have
$\|S\|=\|S^{-1}\|=1$ and by the Bauer-Fike theorem we may conclude
that for any eigenvalue $\lambda$ of $A_N+C_N$, that is for any
isolated eigenvalue $\lambda $ of $A$, there exists an eigenvalue
$\lambda_N$ of $A_N$ such that $|\lambda-\lambda_N|\le
\|C_N\|\le\epsilon/\|\bm v_N\|$. Therefore, $|\lambda_N-\lambda|\to 0$
exponentially with $N$.

\subsection{The general case}
The case of nonsymmetric matrices seems more tricky. In fact, the
Bauer-Fike theorem can be still applied if $A_N$ is diagonalizable but the bound turns into
\[
|\lambda_N-\lambda|\le
\frac{\|\bm w_N\|}{\|\bm v_N\|}
\|Y\|\cdot \|S_N\| \cdot \|S_N^{-1}\|
\]
where $S^{-1}_NA_NS_N=D$ is a diagonal matrix. Therefore, in this case
we need that $A_N$ be diagonalizable and that $\lim_N \| \bm
w_N\|\cdot \|S_N\|\cdot \|S_N^{-1}\|=0$. This condition is satisfied
if, say, the condition number $\|S_N\|\cdot \|S_N^{-1}\|$ is uniformly
bounded from above by a constant.

Unfortunately, the condition number of $S_N$ may grow very fast with
$N$. Think for instance to the tridiagonal matrix
$\hbox{trid}(1/2,0,2)=\widehat D^{-1}\hbox{trid}(1,0,1)\widehat D$
where $\widehat D=\hbox{diag}(1, 2, 2^2,\ldots,2^{N-1})$, having
$S_N=Q_N\widehat D$ as eigenvector matrix with $Q$ orthogonal. Clearly
$\hbox{cond}(S_N)=\hbox{cond}(\widehat D)=2^{N-1}$.

On the other hand, from \eqref{eq:bf} we find that if $\lambda$ is not
eigenvalue of $A_N$, then $(A_N-\lambda I)^{-1}\bm u_N=-\bm v_N$, that
is, $\|(A_N-\lambda I)^{-1}\|\ge \|\bm v_N\|/\|\bm u_N\|\ge \gamma
\epsilon^{-1}$ for some constant $\gamma>0$. This implies that
$\lambda\in\hbox{sp}_{\gamma^{-1} \epsilon}(A_N)$ for any $N>N_0$.

Therefore, we may say that for any eigenvalue $\lambda$ of the QT
matrix $A$ and for any $\epsilon>0$ there exists an integer $N_0$ such
that for any $N\ge N_0$ the matrix $A_N$ has an $\epsilon$-pseudo
eigenvalue $\mu$ equal to $\lambda$.  This fact motivates using the
eigenvalues of $A_N$, for sufficiently large values of $N$, as
starting approximations for Newton's iteration.

\section{Implementation and numerical results}\label{sec:exper}
We have implemented the algorithms described in the previous sections
in Matlab and added them to the CQT-Toolbox of \cite{bmr}.   The
functions allow the computation in high precision arithmetic relying
on the package {\tt Advanpix}, see \url{https://advanpix.com}.
The main functions are \verb+eig_single+ and \verb+eig_all+.  The function
\verb+eig_single+ computes the approximation of a single eigenvalue by relying on Newton's iteration, in both the
Vandermonde and the Frobenius version, starting from a given
approximation $\lambda_0$. The function \verb+eig_all+ computes
approximations to all the eigenvalues starting from the eigenvalues of
the matrix $A_N$ for $N=\gamma\max(h_1,h_2,m+n)$, the value of
$\gamma$ can be optionally changed, by default $\gamma=3$.  The
iterations are halted if the modulus of the difference between two
subsequent approximations is less than $10^3 u$, where $u$ is the
machine precision, and if this value is not less than the value
obtained at the previous step. After that the halting condition is
satisfied, a further Newton step is applied to refine the
approximation. The iterations are halted with the failure flag if
$\hbox{wind}(\lambda_k) \neq \hbox{wind}(\lambda_{k-1})$ for some $k$
or if $|\lambda_k|$ is larger than $\|A\|_\infty$ or if the maximum
number of 20 iterations has been reached.  More information, together
with the description of other auxiliary functions and optional
parameters, can be found at \url{https://numpi.github.io/cqt-toolbox},
while the software can be downloaded at
\url{https://github.com/numpi/cqt-toolbox}.

\subsection*{The tests}
We have performed several tests to validate our
algorithms. Here, we describe the results of the most meaningful
ones. In the following, we denote by {\tt am} and {\tt ap} two vectors
such that ${\tt am}=[a_0,a_{-1},\ldots,a_{-m}]$ and ${\tt
  ap}=[a_0,a_1,\ldots,a_n]$, where $a(z)=\sum_{i=-m}^n a_iz^i$ is the
Laurent polynomial associated with the QT matrix $A=T(a)+E$. We refer to Algorithm
V for the Vandermonde approach and Algorithm F for the Frobenius
approach. The tests have been run on a laptop with Intel I5 CPU and with Matlab version R2021b.

In Test 1 we have set $m=3$ and $n=2$, where ${\tt am}=[0,-1,1,-1]$, ${\tt
  ap}=[0,-1,-1]$.  We have applied two kinds of corrections, namely,
the $20\times 100$ matrix
$E_2$ having null entries except the last column which is equal to
$[1,2,3,\ldots,20]^T$, and the $3\times 100$ matrix $E_1$ having null entries except in the last
column which is equal to $8[1,2,3]^T$.
We refer to these two corrections as Case 1 and Case 2, respectively.

In Test 2 we have set $m = 7$ and $n = 2$
where ${\tt am}\!=\![0, -1, 1, -1, 0, 0, 0, 1]$, ${\tt ap}=[0,-1,-1]$.  We
have applied two kinds of corrections, namely, $E_1$ and $E_2$, where
$E_1$ is the
same as in Test 1, while $E_2$ has size $7\times 100$ with null entries except the last column
which is equal to $8[1,2,3,\ldots,7]^T$.
We refer to these two corrections as Case 1 and Case 2, respectively.

Test 3 has been designed in order to show that the Vandermonde
approach may strongly suffer of numerical instability when the
characteristic equation $a(z)-\lambda=0$ has some clustered roots
that, consequently, are ill conditioned. For this test, we have
constructed $a(z)$ in terms of a Mignotte-like polynomial
\cite{mignotte}. 
More precisely, we set $a(z)=z^{-m}b(z)$, where $b(z)$ is of the form
$b(z)=(10^{-1}+z)^3+10z^{n+m}$. This polynomial has a very tight
cluster of 3 zeros close to $10^{-1}$. In our test we set $m=10$ and
$n=2$ and $E=10^{-5} [0_{12}, I_{12}]$.

In all the three tests the matrix $A$ is not symmetric.

\subsection*{Details on the implementation}
The algorithms have been applied
in the double precision floating point arithmetic.  The basins of
attraction have been constructed as follows. A generic point in the
picture, corresponding to the complex number $\lambda_0$ has been coloured
with a colour depending on the limit of the sequence generated by
fixed point iteration $\lambda_{k+1}=g(\lambda_k)$ for $k\ge 0$. Different colours, randomly
generated, have been used for different limits. Different levels of
gray have been used to denote that the iteration has been halted with
no convergence. The colour green has been used for the values $\lambda_0$
belonging to a continuous set of eigenvalues.

\subsection*{The results}
In the figures where the eigenvalues are reported,
red circles indicate the
eigenvalues of the finite section $A_N$, blue
dots represent isolated eigenvalues of $A$, while red circles containing a green dot represent eigenvalues of $A_N$ that belong to a continuous set of eigenvalues. The light blue curve denotes the set $a(\mathbb T)$.
 In the figures displaying the basins of attraction, the light green area indicates a continuous set of eigenvalues.
For this set of figures, Algorithm V has been applied.

Figure \ref{fig:1_1} displays the eigenvalues of
$A+E_1$, for the matrix of Test 1, and the basins of attraction of
Newton's iteration, together with a zoom of a specific area.  

Figure \ref{fig:1_2} displays the analogous images for the matrix
$A+E_2$ of Test 1.   Here, it is interesting to observe the existence of
a connected component formed by a continuous set of eigenvalues
denoted by a green triangle-shaped figure. Observe also that the corresponding red circles in this component contain a
green dot.

\begin{figure}
\centering
\includegraphics[scale=0.34]{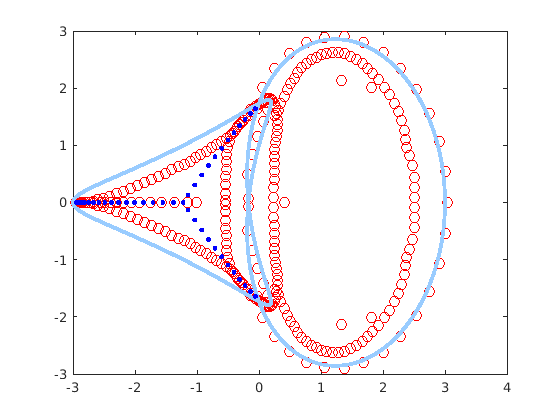}~ 
\includegraphics[scale=0.34]{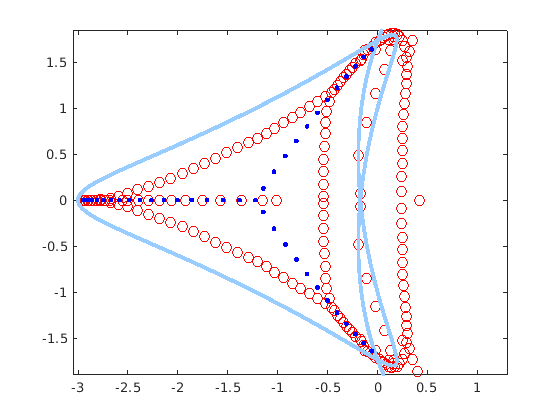}\\ 
\includegraphics[scale=0.34]{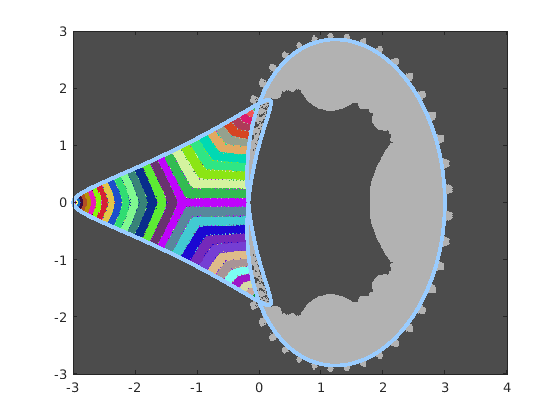}~
\includegraphics[scale=0.34]{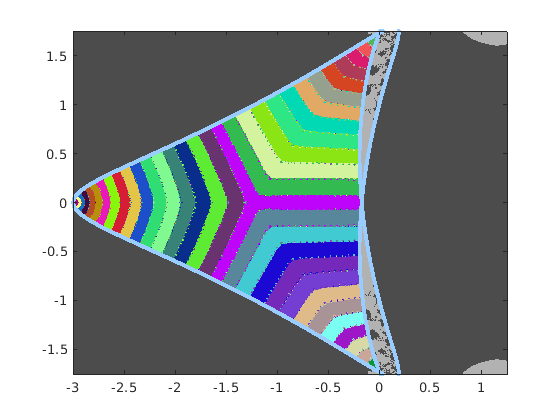}
\caption{\footnotesize
  Test 1, Case 1: Eigenvalues of the QT matrix $A$ (blue dots) and of
  the finite section $A_N$ (red circles), together with the basins of
  attraction for Newton's iteration computed by Algorithm V. On the right the zoom of a
  portion. }\label{fig:1_1}
\end{figure}

\begin{figure}
\centering
\includegraphics[scale=0.34]{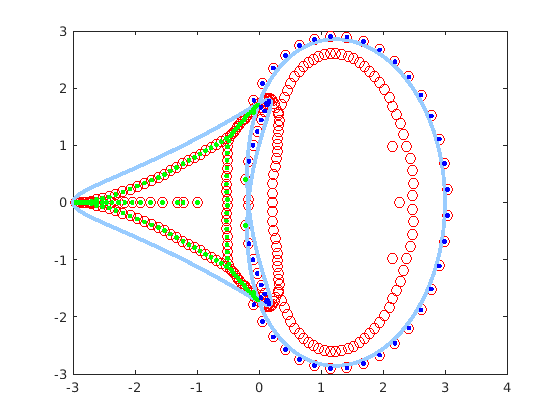}~
\includegraphics[scale=0.34]{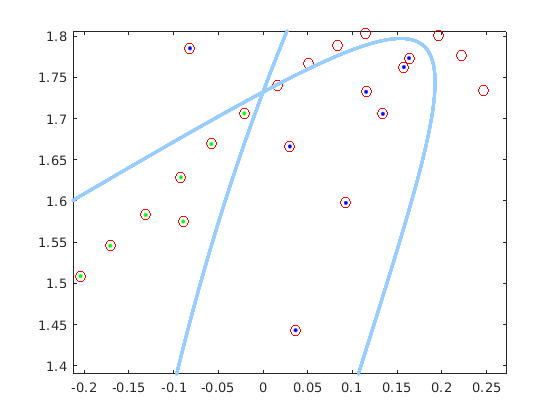}\\ 
\includegraphics[scale=0.34]{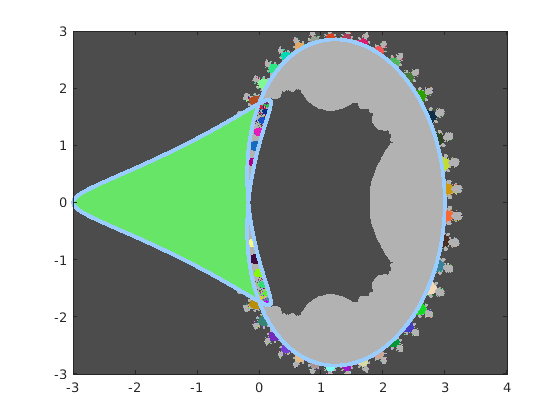}~
\includegraphics[scale=0.34]{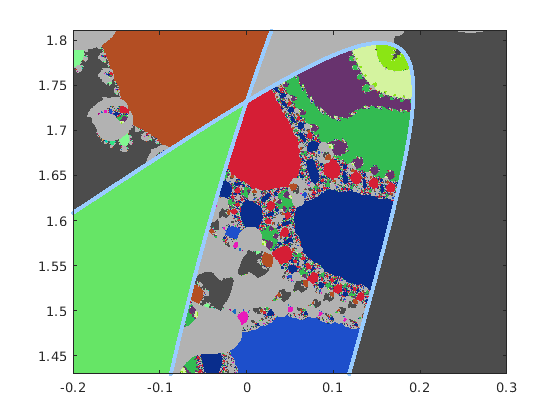}
\caption{\footnotesize
  Test 1, Case 2: Eigenvalues of the QT matrix $A$ (blue dots) and of
  the finite section $A_N$ (red circles), together with the basins of
  attraction for Newton's iteration computed by Algorithm V. On the right the zoom of a
  portion. }\label{fig:1_2}
\end{figure}

The smallest value of $N_0$ for which the number of computed
eigenvalues is constant for $N\ge N_0$ is $N_0=400$ for
the Case 1, while it is $N_0=200$ for the Case 2.  In both cases, the
geometry of the basins of attraction, together with the distribution of
the eigenvalues of $A_N$, explains why Newton's iteration converges to
all the eigenvalues, when starting from the eigenvalues of $A_N$ for a
quite small value of $N$, even though the latter eigenvalues are far
from the eigenvalues of $A$. This latter property is more evident in
 Case 1, where several blue dots are not contained inside red circles, see
Figure \ref{fig:1_1}, zoomed part.

The number of iterations to arrive at convergence is quite small and is the same for both algorithms.
Namely, concerning Case 1, it ranges from 3 to 18 with the avergae value of 7.5; concerning Case 2, it ranges from 3 to 10 with average 3.3. 

Another interesting issue to investigate, independently of the
algorithm used, is to analyze how large must be $N$ in order that the
eigenvalues of $A_N$ approximate all the eigenvalues of $A$ within the machine precision $u=\tt 2.22e-16$ so that no step of Newton's
iteration would be necessary. It turns out that for the Test 1, Case
1, almost all the eigenvalues are well approximated already for $N=800$, while
there are few eigenvalues that require a pretty larger size. Table
\ref{tab:dist1_1} shows a few significant cases. Typically, the
eigenvalues closest to the light blue curve are the ones that need a
large value of the truncation level $N$ to be properly approximated by
a corresponding eigenvalue of $A_N$.  For instance, from Table
\ref{tab:dist1_1} it turns out that $N=3200$ is not enough to
approximate the rightmost eigenvalue. Even $N=6400$ does not provide a
full accuracy approximation.  A similar situation holds for the Case
2.

\begin{table}\centering
\begin{tabular}{c|cccccc}
$\lambda~\backslash~ N$                   & 200   & 400   &  800    & $1600$ & $3200$ & $6400$ \\ 
\hline
-4.0e-01$\pm$1.2e+00i & 4.1e-04 & 1.3e-08 &   -- &                      
\\
-3.1e-01$\pm$1.3e+00i & 3.0e-03 & 3.9e-06 &  5.3e-12  & --                        
\\
-2.2e-01$\pm$1.5e+00i & 5.4e-03 & 6.1e-05 &  5.9e-09  & --                           
\\
-1.4e-01$\pm$1.6e+00i & 2.6e-02 & 2.0e-03 &  2.7e-05  & 1.9e-10        & --           
\\
-5.9e-02$\pm$1.6e+00i & 6.5e-02 & 3.5e-02 &  1.1e-02  & 3.5e-04        & 9.2e-07        & 3.8e-13      
\end{tabular}\caption{\footnotesize Test1, Case 1: Distances of some eigenvalues of $A$ from the closest eigenvalue of $A_N$ for different values of $N$. A ``--'' denotes a value below 1.e-15.}
\label{tab:dist1_1}\end{table}

Concerning the accuracy of approximation, Figure \ref{fig:errors1}
shows the relative errors of approximating the eigenvalues of $A$ with
the Vandermonde approach (blue circle) and with the Frobenius approach
(red cross) in the two cases of Test 1. Here, the eigenvalues have
been sorted according to the real part.  The relative errors have been
obtained by comparing the eigenvalues computed in the standard
floating point arithmetic with those computed in the quadruple
precision relying on Advanpix.  Observe that the results obtained by
the Frobenius version are generally more accurate than the ones
obtained with the Vandermonde version.

\begin{figure}\centering
\includegraphics[scale=0.35]{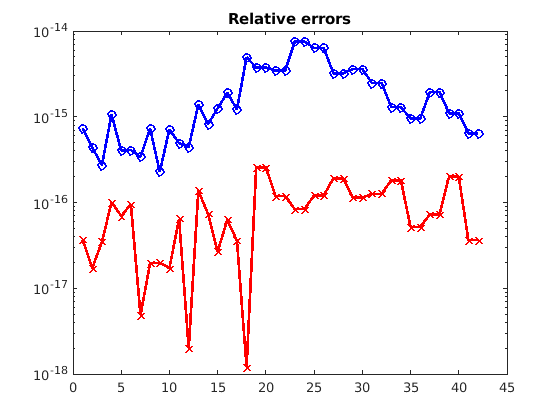}
\includegraphics[scale=0.35]{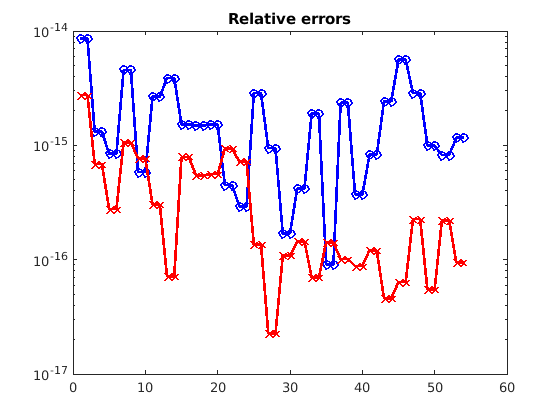}
\caption{\footnotesize Test 1. Relative errors in each eigenvalue computed with Algorithm V (blue circle) and with Algorithm F (red cross). Case 1 and case 2 on the left and on the right, respectively. Eigenvalues are sorted with respect to the real part.}\label{fig:errors1}
\end{figure}

Finally, concerning the CPU time, the two algorithms have similar
performances even though, for this test, Algorithm F generally
requires a double time.

Test 2, Case 1, points out in a more evident manner that the eigenvalues of $A$
which are close to the curve $a(\mathbb T)$ can be hardly approximated
by the eigenvalues of a finite section $A_N$ of $A$, unless $N$ is
extremely large. In fact, as clearly shown in Figure \ref{fig:2} 
and in the zoomed areas, out of the 8 eigenvalues of $A$, there is a
group of few eigenvalues that lie very close to the light blue
curve. In particular, the second (from the left) eigenvalue and the
last one.  The distances of these eigenvalues to the closest
eigenvalue of $A_N$ for different values of $N$ are reported in Table
\ref{tab:d2}. It turns out that in order to approximate such
eigenvalues within the machine precision $u$ without applying Newton's
iteration, one would need truncation levels larger than $1.6$
millions, whereas Newton's iteration converges quickly just starting
from the eigenvalues of $A_N$, with $N=3200$.

\begin{figure}
\centering
\includegraphics[scale=0.34]{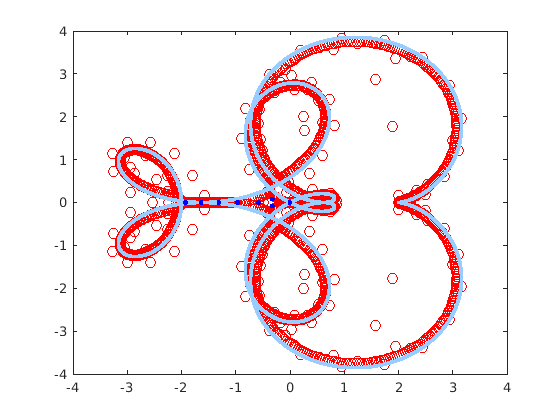}~ \includegraphics[scale=0.34]{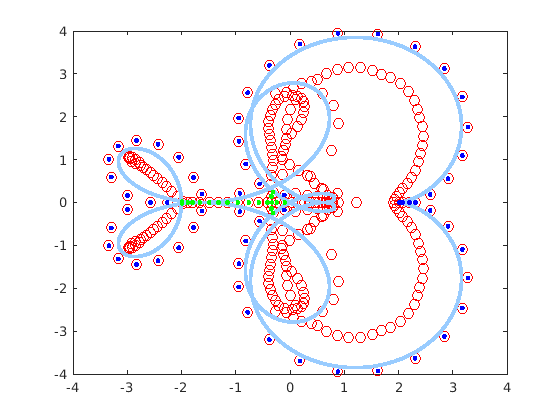}\\
\includegraphics[scale=0.33]{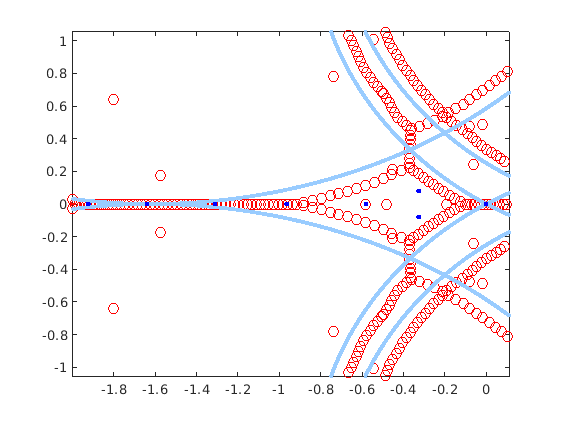}\includegraphics[scale=0.33]{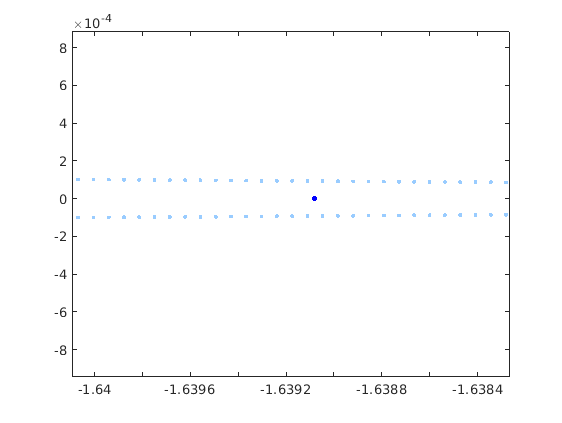}
\includegraphics[scale=0.33]{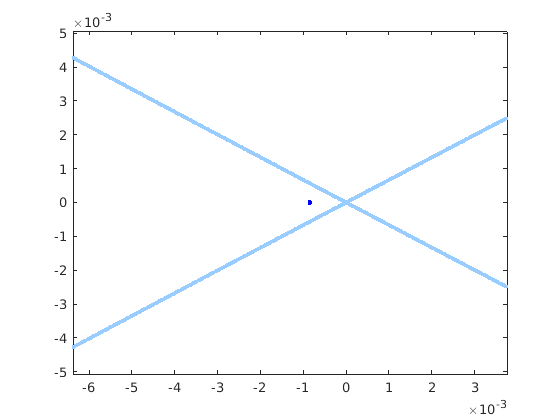}
\caption{\footnotesize
  Test 2, In the first line the eigenvalues of Case 1 (left) and Case 2 (right) are displayed with a blue dot. In the second line, some portion of the domain where the eigenvalues of Case 1 are located are displayed; more specifically, from
  the left, the area with all the eigenvalues, the second leftmost eigenvalue, and the last rightmost
  eigenvalue are zoomed, respectively. }\label{fig:2}
\end{figure}


Also in this test, the number of iterations required by Algorithm V
and Algorithm F is the same. Namely, for Case 1 it ranges between 5 and 12 with average value 7.25, for Case 2 it ranges between 2 and 4 with average value 3.0. Algorithm F turns out to be
more accurate than Algorithm V as shown in Figure
\ref{fig:errors2}.

\begin{figure}\centering
\includegraphics[scale=0.35]{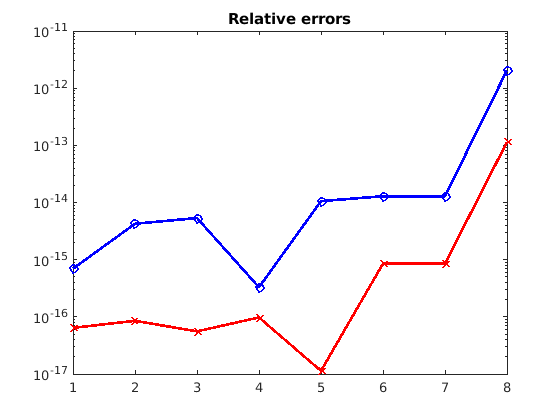}~ 
\includegraphics[scale=0.35]{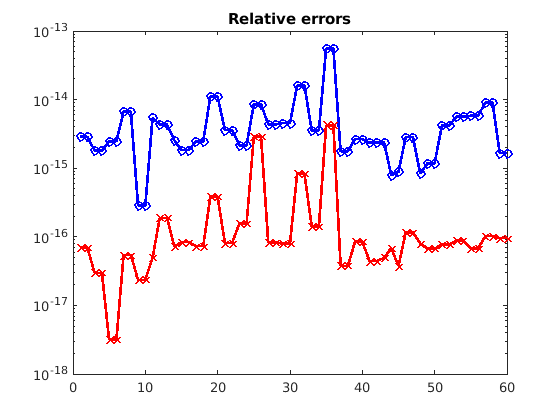}
\caption{\footnotesize Test 2.
  Relative errors in each eigenvalue computed with Algorithm V
  (blue circle) and with Algorithm F (red cross). Case 1 and case 2 on the left and on the right, respectively. Eigenvalues are
  sorted with respect to the real part.}\label{fig:errors2}
\end{figure}

\begin{table}\centering
\begin{tabular}{c|ccccccc}
$\lambda~\backslash~ N$    & $400$ & $1600$ & $6400$ & $25600$ & $102400$ & $409600$& $1638400$\\
\hline
-1.9   &1.4e-01 &  8.3e-02 &  4.1e-05  & -- \\ 
-1.6   &7.6e-01 &  2.3e-01 &  9.8e-02  & 5.8e-02 &  2.9e-03 &  2.6e-03 &  1.4e-07\\
-1.3   &4.2e-01 &  1.2e-01 &  1.5e-04  & -- \\ 
-9.6e-01   &1.6e-01 &  5.2e-06 &  --\\ 
-5.8e-01   &3.0e-03 &  7.3e-11 &    --    \\
-8.5e-04   &6.8e-02 &  1.0e-01 &  9.2e-02 &  7.8e-03 &  2.9e-04 &  5.3e-13& --\\
 \end{tabular}\caption{\footnotesize Test 2, Case 1: Distances of the real eigenvalues of $A$ from the closest eigenvalue of $A_N$ for different values of $N$. A ``--'' denotes a value below 1.e-15.}
 \label{tab:d2}
\end{table}

Concerning Test 3, the matrix
$A$ has a set $\mathcal S$ of 22 eigenvalues, shown in Figure
\ref{fig:test3}, that can be grouped into 3 subsets $\mathcal S_1$,
$\mathcal S_2$, $\mathcal S_3$. The subset $\mathcal S_1$ is formed by
4 entries of modulus in the range $[0.25,1.7]$, while $\mathcal S_2$
and $\mathcal S_3$ are formed by 9 entries of modulus roughly 20, and
30, respectively. Recall that the Mignotte-like polynomial $z^ma(z)$ has a tight cluster formed by three ill-conditioned zeros. 
For $\lambda\in\mathcal S$ the polynomial $z^m(a(z)-\lambda)$ still has a
cluster of ill-conditioned zeros, where the cluster is tighter and
consequently the zeros are more ill-conditioned the smaller is $|\lambda|$.
This explains why the errors of the algorithm based on the Vandermonde
formulation are much higher in the leftmost part of the graph shown in
Figure \ref{fig:test3}. 

\begin{figure}[ht]\centering
\includegraphics[scale=0.33]{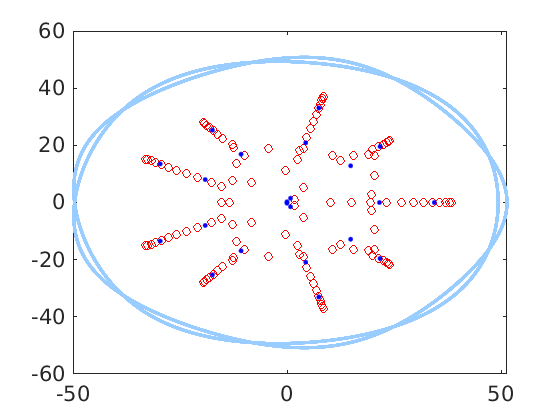}
\includegraphics[scale=0.33]{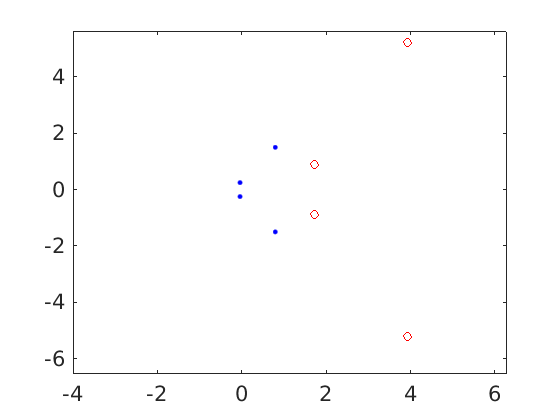}
\includegraphics[scale=0.33]{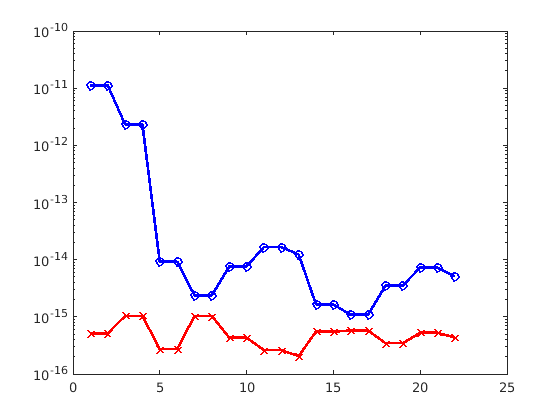}
\caption{\footnotesize Test 3. From the left: Geometry of the eigenvalues with a zoom of the cluster computed by Algorithm F; relative errors for each eigenvalue computed by Algorithm V (blue circle) and by Algorithm F (red cross), where eigenvalues are sorted by increasing modulus.}\label{fig:test3}
\end{figure}

\section{Conclusions and open problems}
The problem of computing the eigenvalues of a QT matrix has been
reformulated as a nonlinear eigenvalue problem. Newton's iteration has
been analyzed for this task both in the Vandermonde version and in the
Frobenius version. As initial approximation for starting the iteration 
we use the eigenvalues of the truncated matrix $A_N$.
 Numerical experiments
show the effectiveness of our approach. The algorithm based on the
Frobenius formulation turned out to be more accurate even though
slightly slower.  Approximating all the eigenvalues to the machine
precision directly from the eigenvalues of the truncated matrix $A_N$, without
using Newton's iteration, is shown to be infeasible due to the huge
values needed for $N$. A Matlab implementation of the algorithm has been provided
and the software has been included in the CQT-toolbox of \cite{bmr}.

In order to make the software more robust and effective we plan to
provide an optimized implementation of polynomial spectral
factorization relying on the algorithms of \cite{BH13} and
\cite{BH14}. Another important issue is to find theoretical estimates
of the truncation parameter $N$ that guarantees the approximation to
all the eigenvalues of $A$, starting from those of $A_N$.  
Other approaches to solving the nonlinear eigenvalue problem, say the ones based on rational approximation, could be the subject of subsequent research.

\bibliographystyle{abbrv}


\begin{thebibliography}{10}

\bibitem{barnett}
S.~Barnett.
\newblock {\em Polynomials and linear control systems}, volume~77 of {\em
  Monographs and Textbooks in Pure and Applied Mathematics}.
\newblock Marcel Dekker, Inc., New York (1983).

\bibitem{bfgm}
D.~A. Bini, G.~Fiorentino, L.~Gemignani, and B.~Meini.
\newblock Effective fast algorithms for polynomial spectral factorization.
\newblock {\em Numer. Algorithms}, 34(2-4), 217--227 (2003).

\bibitem{mean1}
D.~A. Bini, B.~Iannazzo, and J.~Meng.
\newblock {A}lgorithms for {A}pproximating {M}eans of {S}emi-infinite
  {Q}uasi-{T}oeplitz {M}atrices.
\newblock In B.~F. Nielsen~F., editor, {\em Geometric Science of Information,
  GSI 2021}, volume 12829 of {\em Lecture Notes in Computer Science}, pages
  405--414. Springer (2021).

\bibitem{mean2}
D.~A. Bini, B.~Iannazzo, and J.~Meng.
\newblock Geometric means of quasi-{T}oeplitz matrices.
\newblock arXiv preprint.
\newblock (2021).

\bibitem{blm:book}
D.~A. Bini, G.~Latouche, and B.~Meini.
\newblock {\em Numerical methods for structured {M}arkov chains}.
\newblock Numerical Mathematics and Scientific Computation. Oxford University
  Press, New York (2005).

\bibitem{bmm}
D.~A. Bini, S.~Massei, and B.~Meini.
\newblock Semi-infinite quasi-{T}oeplitz matrices with applications to {QBD}
  stochastic processes.
\newblock {\em Math. Comp.}, 87(314), 2811--2830 (2018).

\bibitem{bmmr}
D.~A. Bini, S.~Massei, B.~Meini, and L.~Robol.
\newblock On quadratic matrix equations with infinite size coefficients
  encountered in {QBD} stochastic processes.
\newblock {\em Numer. Linear Algebra Appl.}, 25(6), 2128, 12 (2018).

\bibitem{bmmr:sicomp}
D.~A. Bini, S.~Massei, B.~Meini, and L.~Robol.
\newblock A computational framework for two-dimensional random walks with
  restarts.
\newblock {\em SIAM J. Sci. Comput.}, 42(4), A2108--A2133 (2020).

\bibitem{bmr}
D.~A. Bini, S.~Massei, and L.~Robol.
\newblock Quasi-{T}oeplitz matrix arithmetic: a {MATLAB} toolbox.
\newblock {\em Numerical Algorithms}, 81(2), 741--769 (2019).

\bibitem{bm:numermath}
D.~A. Bini and B.~Meini.
\newblock On the exponential of semi-infinite quasi-{T}oeplitz matrices.
\newblock {\em Numer. Math.}, 141(2), 319--351 (2019).

\bibitem{bmm20}
D.~A. Bini, B.~Meini, and J.~Meng.
\newblock Solving quadratic matrix equations arising in random walks in the
  quarter plane.
\newblock {\em SIAM J. Matrix Anal. Appl.}, 41(2), 691--714 (2020).

\bibitem{BG:book2000}
A.~B\"{o}ttcher and S.~M. Grudsky.
\newblock {\em Toeplitz matrices, asymptotic linear algebra, and functional
  analysis}.
\newblock Birkh\"{a}user Verlag, Basel (2000).

\bibitem{BG:book2005}
A.~B\"{o}ttcher and S.~M. Grudsky.
\newblock {\em Spectral properties of banded {T}oeplitz matrices}.
\newblock Society for Industrial and Applied Mathematics (SIAM), Philadelphia,
  PA (2005).

\bibitem{BH13}
A.~B\"{o}ttcher and M.~Halwass.
\newblock A {N}ewton method for canonical {W}iener-{H}opf and spectral
  factorization of matrix polynomials.
\newblock {\em Electron. J. Linear Algebra}, 26, 873--897 (2013).

\bibitem{BH14}
A.~B\"{o}ttcher and M.~Halwass.
\newblock Wiener-{H}opf and spectral factorization of real polynomials by
  {N}ewton's method.
\newblock {\em Linear Algebra Appl.}, 438(12), 4760--4805 (2013).

\bibitem{BS:book}
A.~B\"{o}ttcher and B.~Silbermann.
\newblock {\em Introduction to large truncated {T}oeplitz matrices}.
\newblock Universitext. Springer-Verlag, New York (1999).

\bibitem{dangelo}
J.~P. D'Angelo.
\newblock {\em Several complex variables and the geometry of real
  hypersurfaces}.
\newblock Studies in Advanced Mathematics. CRC Press, Boca Raton, FL (1993).
\bibitem{gander-2021}
W.~Gander.
\newblock New algorithms for solving nonlinear eigenvalue problems.
\newblock {\em Comput. Math. Math. Phys.}, 61(5), 761--773 (2021).

\bibitem{gs2:book}
C.~Garoni and S.~Serra-Capizzano.
\newblock {\em Generalized locally {T}oeplitz sequences: theory and
  applications. {V}ol. {I}}.
\newblock Springer, Cham (2017).

\bibitem{gs1:book}
C.~Garoni and S.~Serra-Capizzano.
\newblock {\em Generalized locally {T}oeplitz sequences: theory and
  applications. {V}ol. {II}}.
\newblock Springer, Cham  (2018).

\bibitem{feast-2018}
B.~Gavin, A.~Mi\polhk{e}dlar, and E.~Polizzi.
\newblock F{EAST} eigensolver for nonlinear eigenvalue problems.
\newblock {\em J. Comput. Sci.}, 27, 107--117 (2018).

\bibitem{gt-2017}
S.~G\"{u}ttel and F.~Tisseur.
\newblock The nonlinear eigenvalue problem.
\newblock {\em Acta Numer.}, 26, 1--94 (2017).

\bibitem{hp-2020}
M.~E. Hochstenbach and B.~Plestenjak.
\newblock Computing several eigenvalues of nonlinear eigenvalue problems by
  selection.
\newblock {\em Calcolo}, 57(2), Paper No. 16, 25 (2020).

\bibitem{jackson}
J.~R. Jackson.
\newblock Networks of waiting lines.
\newblock {\em Operations Res.}, 5, 518--521 (1957).

\bibitem{jie}
H.-M. Kim and J.~Meng.
\newblock Structured perturbation analysis for an infinite size
  quasi-{T}oeplitz matrix equation with applications.
\newblock {\em BIT Numerical Mathematics}, 61, 859--879 (2021).

\bibitem{lr:book}
G.~Latouche and V.~Ramaswami.
\newblock {\em Introduction to matrix analytic methods in stochastic modeling}.
\newblock ASA-SIAM Series on Statistics and Applied Probability. SIAM,
  Philadelphia, PA (1999).

\bibitem{mignotte}
M.~Mignotte.
\newblock Some useful bounds.
\newblock In {\em Computer algebra}, pages 259--263. Springer, Vienna (1983).

\bibitem{neuts}
M.~F. Neuts.
\newblock {\em Matrix-geometric solutions in stochastic models: An algorithmic
  approach}, volume~2 of {\em Johns Hopkins Series in the Mathematical
  Sciences}.
\newblock Johns Hopkins University Press, Baltimore, Md. (1981).

\bibitem{graeffe}
A.~Ostrowski.
\newblock {Recherches sur la méthode de Graeffe et les zéros des polynomes et  des séries de Laurent}.
\newblock {\em Acta Mathematica}, 72, 99 -- 155 (1940).

\bibitem{ozawa19}
T.~Ozawa.
\newblock Stability condition of a two-dimensional qbd process and its
  application to estimation of efficiency for two-queue models.
\newblock {\em Performance Evaluation}, 130, 101 -- 118 (2019).

\bibitem{ozawa}
T.~Ozawa.
\newblock Asymptotic properties of the occupation measure in a multidimensional
  skip-free {M}arkov-modulated random walk.
\newblock {\em Queueing Syst.}, 97(1-2), 125--161 (2021).

\bibitem{leonardo}
L.~Robol.
\newblock {R}ational {K}rylov and {ADI} iteration for infinite size
  quasi-Toeplitz matrix equations.
\newblock {\em Linear Algebra Appl.}, 604, 210--235 (2020).

\bibitem{Schechter}
M.~Schechter.
\newblock Basic theory of {F}redholm operators.
\newblock {\em Ann. Scuola Norm. Sup. Pisa Cl. Sci. (3)}, 21:261--280 (1967).

\end{thebibliography}

\end{document}